\documentclass[a4paper,11pt]{article}

\usepackage{amsmath,amsfonts,amsthm,amssymb}  
\usepackage[affil-it]{authblk}
\usepackage{booktabs}
\usepackage{cite}
\usepackage[top=2.5cm,bottom=2.5cm,left=3cm,right=3cm]{geometry} 
\usepackage{graphicx}
\usepackage{mathtools}
\usepackage[final,nopatch=footnote]{microtype}
\usepackage{titlesec}
\usepackage{xcolor}
\usepackage{stmaryrd}

\usepackage[T1]{fontenc}

\usepackage[colorlinks=true,allcolors=blue]{hyperref}

\usepackage{setspace}
\theoremstyle{definition}

\theoremstyle{plain}
\newtheorem{conjecture}{Conjecture}

\newtheorem{lemma}{Lemma}
\newtheorem*{proposition*}{Proposition}
\newtheorem{proposition}{Proposition}
\newtheorem{theorem}{Theorem}

\theoremstyle{remark}

\newcommand{\inner}[2]{\langle #1,\,#2 \rangle}

\renewcommand*{\Im}{\mathop{\mathrm{Im}} \nolimits}

\providecommand{\ap}[1]{\ensuremath{^\text{#1}}}
\providecommand{\ped}[1]{\ensuremath{_\text{#1}}}

\providecommand{\Z}{\mathbb Z}
\providecommand{\Q}{\mathbb Q}
\providecommand{\R}{\mathbb R}
\providecommand{\C}{\mathbb C}

\renewcommand*{\sl}{\mathop{\mathfrak{sl}} \nolimits}
\newcommand*{\SL}{\mathop{\mathrm{SL}} \nolimits}

\newcommand*{\Nil}{\mathop{\mathrm{Nil}} \nolimits}
\newcommand*{\Sol}{\mathop{\mathrm{Sol}} \nolimits}

\DeclareMathOperator{\Tr}{Tr}

\DeclareMathOperator{\coker}{coker}

\DeclareMathOperator{\CS}{CS}

\titleformat{\section}[hang]
    {\large\bfseries}
    {\thesection}
    {1em}
    {}
    []

\titleformat{\subsection}[hang] 
    {\bfseries}
    {\thesubsection}
    {1em}
    {}
    []

\titleformat{\subsubsection}[hang] 
    {\bfseries}
    {\thesubsubsection}
    {1em}
    {}
    []

\titleformat{\paragraph}[runin]
    {\normalfont}
    {\theparagraph}
    {1em}
    {\itshape}
    []

\usepackage{tikz}
\usetikzlibrary{knots}

\title{On Quantum Modularity for \\ Geometric 3-Manifolds}

\author[1]{Pavel Putrov}
\affil[1]{ICTP, Strada Costiera 11, Trieste 34151, Italy}

\author[2]{Ayush Singh}
\affil[2]{SISSA, Via Bonomea 265, Trieste 34136, Italy}

\date{}

\begin{document}

\maketitle

\begin{abstract}
\noindent
The quantum modularity conjecture, first introduced by Don Zagier, is a general statement about a relation between $\mathfrak{sl}_2$ quantum invariants of links and 3-manifolds at roots of unity related by a modular transformation. In this note we formulate a strong version of the conjecture for Witten--Reshetikhin--Turaev invariants of closed geometric, not necessarily hyperbolic, 3-manifolds. This version in particular involves a geometrically distinguished $SL(2,\mathbb{C})$ flat connection (a generalization of the standard hyperbolic flat connection to other Thurston geometries) and has a statement about the integrality of coefficients appearing in the modular transformation formula. We prove that the conjecture holds for Brieskorn homology spheres and some other examples. We also comment on how the conjecture relates to a formal realization of the $\mathfrak{sl}_2$ quantum invariant at a general root of unity as a path integral in analytically continued $SU(2)$ Chern--Simons theory with a rational level.

\end{abstract}

\tableofcontents

\section{Introduction and summary}
\label{sec:intro}

Thurston's geometrization conjecture \cite{thurston1982three}, also known as Perelman's theorem \cite{perelman2002entropyformularicciflow,perelman2003ricciflowsurgerythreemanifolds,perelman2003finiteextinctiontimesolutions}, provides an intimate relation between 3-manifold topology and Riemannian structures. It states that any compact oriented 3-manifold can be cut along 2-spheres and 2-tori so that each piece locally has one of the eight canonical geometries (listed in the first column of Table~\ref{tab:geometries}). 

An efficient tool to study the topology of 3-manifolds is provided by topological quantum field theories (TQFTs). A topological quantum field theory, in particular, provides diffeomorphism invariants of 3-manifolds that behave in a nice way under cutting and gluing. There are two main mathematical constructions of 3d TQFTs. The first one is by Reshetikhin and Turaev \cite{RT91,turaev1992modular}, which is inspired by the physics construction of Witten \cite{Witten:1988hf}. As an input, it takes a modular tensor category (that is, a braided fusion category with a non-degenerate $S$-matrix). One then defines invariants of closed 3-manifolds via their Dehn surgery representation. The invariant of closed 3-manifolds can then be extended to a 3d TQFT. The second construction is by Turaev and Viro \cite{turaev1992state}. As an input, it takes a spherical fusion category. One then defines invariants of 3-manifolds in terms of their triangulation. The two invariants turn out to be closely related: Turaev--Viro invariant for a given spherical function category is equal to Reshetikhin--Turaev invariant for the Drinfeld center of that category. In particular, if in the Turaev--Viro construction as an input one takes a modular tensor category (which is a special type of spherical fusion category), then the resulting 3-manifold invariant is the absolute value of the Witten--Reshetikhin--Turaev (WRT) invariant construction from the same category. Both of these constructions are of a combinatorial nature and, a priori, have nothing to do with geometric structures.  

Nevertheless, it is natural to ask if the eight Thurston geometric structures can be, in one way or another, recovered from those topological invariants. One may expect this especially when the input category for the WRT invariant is the modular tensor category of finite-dimensional representations of the quantum group $\bar{U}_\xi(\sl_2$) for $\xi$ being a primitive $r$-th root of unity. In the case when $\xi=e^{\frac{2\pi i}{r}}$ for an integer $r$, the corresponding TQFT can be understood as $SU(2)$ Chern--Simons gauge theory. Formally, its analytically continued version can be related to the 3d theory of gravity \cite{Witten:1988hc,Witten:2010cx}.

In the case of hyperbolic knots, that is, knots with complements admitting hyperbolic Thurston geometry, such a relation between quantum invariants and geometric structures is provided by Kashaev's \textit{volume conjecture} \cite{kashaev1997hyperbolic} reformulated in terms of the colored Jones polynomial by Murakami and Murakami \cite{murakami2001colored}. It states that the asymptotics of the absolute value of the colored Jones polynomial of a knot in a certain regime is given by the volume of the knot complement for the canonical hyperbolic metric. A slightly more refined version of this conjecture states that \cite{murakami2002kashaev}:
\begin{equation}
    \lim_{r\rightarrow \infty} \frac{\log J_r(K;e^\frac{2\pi i}{r})}{2\pi i r}=\CS[A_*]\mod 1.
\end{equation}
On the left-hand side, $J_r(K;\xi)$ is the $r$-th colored Jones polynomial of the knot $K\in S^3$ normalized to 1 on the unknot. On the right-hand side, $\CS[A_*]$ is the Chern--Simons functional of the so-called $SL(2,\C)$ geometric flat connection representing the homomorphism $\pi_1(S^3\setminus K)\rightarrow SL(2,\C)$ that can be defined in the following way. Since the knot has hyperbolic complement, we have $S^3\setminus K\cong \mathbb{H}^3/\Gamma$, where $\mathbb{H}^3$ is the 3-dimensional hyperbolic space and $\Gamma\cong \pi_1(S^3\setminus K)$ is a discrete subgroup of orientation-preserving isometries $\mathrm{Isom}_+(\mathbb{H}^3)\cong PSL(2,\C)$. A lift (which always exists) of the inclusion $\Gamma \hookrightarrow PSL(2,\C)$ to $SL(2,\C)$ provides the above homomorphism. The relation to the hyperbolic volume appears through the fact that 
\begin{equation}
    \Im \CS[A_*]=-\frac{\mathrm{Vol}(S^3\setminus K)}{4\pi^2}.   
\end{equation}

The volume conjecture for the colored Jones polynomial can be further refined to \textit{quantum modularity} \cite{zagier2010quantum,Garoufalidis:2021lcp}. It can be formulated as a statement about the expression of the $r\rightarrow\infty$ asymptotics of the $r$-th colored Jones polynomial evaluated at a general $r$-th primitive root of unity $\xi\coloneqq e^{\frac{2\pi is}{r}}$ through the $s$-th colored Jones polynomial evaluated at $\tilde{\xi} \coloneqq e^{-\frac{2\pi ir}{s}}$.

An analogue of the volume conjecture for closed hyperbolic 3-manifolds first appeared in the work of Chen and Yang \cite{Chen:2015wfa}, where it was formulated as the appearance of the hyperbolic volume in the leading order of the asymptotics of Witten--Reshetikhin--Turaev and Turaev--Viro invariants for $\xi=e^{\frac{4\pi i}{r}}$ with odd $r$. This conjecture was then refined to a quantum modularity statement for hyperbolic 3-manifolds by Wheeler \cite{Wheeler:2023cht} as a statement about the $r\rightarrow\infty$ asymptotics of the WRT invariant for general roots of unity $\xi=e^{\frac{2\pi i s}{r}}$. In this paper, we present a generalized version of that conjecture, which is formulated universally for geometric 3-manifolds, not just hyperbolic ones. In the simplest setting, it can be stated as follows (with some details elaborated in the main text).
\begin{conjecture}
\label{conj:zhs}
    Let $M\cong N/\pi_1(M)$ be a geometric integer homology sphere with the model geometry $N$ and $W_M(\xi)$ be its properly normalized WRT invariant for a primitive root of unity $\xi$ of odd order. Specifically, let $\xi=e^{\frac{2\pi is}{r}}$ for some $r\in 2\Z_{\geq 1}+1$ and $s\in 4\Z+1$. Then
    \begin{enumerate}
        \item the WRT invariant has the asymptotic expansion of the following form for as $r
            \to +\infty$ along integers coprime with $s$:
            \begin{equation}
                W_M(\xi) \simeq \sum_{A \in \pi_0(\mathrm{Hom}(\pi_1(M),
                SL(2, \C)))} e^{2\pi i \frac{r}{s} \CS[A]} \,P_A(\tilde \xi)\,
                I_A\left(\frac{s}{r}\right)
                \label{intro-conj-exp}
            \end{equation}
            where $\tilde\xi = e^{-\frac{2\pi i r}{s}}$, $P_A(\tilde \xi)\in\Z[\tilde{\xi}]$, and $I_A(s/r)\in (s/r)^{\delta_A/2}\C \llbracket s/r \rrbracket$ for some $\delta_A\in \Z$.
        \item for the geometric flat connection $A_*$ representing the map $\pi_1(M)\rightarrow SL(2,\C)$ obtained by the lift of the composition of the canonical inclusion $\pi_1(M)\hookrightarrow \mathrm{Isom}_+(N)$ and a certain map $\mathrm{Isom}_+(N)\rightarrow PSL(2,\C)$ we have
            \begin{equation}
                P_{A_*}(\xi) = \xi^\delta W_M(\xi) + C_N,
                \label{zhs-P-W}
            \end{equation}
            where $\delta\in \Z$ and 
            \begin{equation}
             C_N=\left\{
                \begin{array}{cc}
                     -1 & N=S^3, \\
                     0 & \text{otherwise.}
                \end{array}
             \right.
            \end{equation}

    \end{enumerate}
\end{conjecture}
A few remarks are in order: 

\begin{itemize}
    
    \item Different versions of quantum modularity property of WRT invariant (as well as the so-called $\hat{Z}$-invariants---certain analytic continuations of WRT invariants from roots of unity $\xi$ to generic complex $q$ \cite{Gukov:2016gkn,Gukov:2017kmk}) for general, not necessarily hyperbolic, 3-manifolds have been known for a long time (see e.g. \cite{LZ99,Hikami05,Hikami05Brieskorn,Hikami06,Hikami06Spherical,Hikami11,Gukov:2016njj,Gukov:2016gkn,Gukov:2017kmk,Cheng:2018vpl,bringmann2020quantum,Cheng:2023row,Cheng:2024vou}). The main novelty here is the geometric interpretation of the special flat connection $A_*$ which recovers the WRT invariant for non-hyperbolic geometric 3-manifolds. The factors $P_A(\tilde{\xi})$ in the asymptotic expansion can be interpreted as components of the vector-valued quantum modular form. Their conjectural integrality $P_A(\tilde{\xi})\in \Z[\tilde{\xi}]$ is consistent with the integrality of the WRT invariant for integer homology spheres \cite{murakami1994quantum,lawrence2021witten,habiro2002quantum} which is a consequence of the existence of the universal invariant valued in Habiro ring---a certain completion of $\Z[q,q^{-1}]$---the evaluation of which at $q=\xi$ gives the WRT invariant. We have numerical evidence which suggests that there exist elements of the Habiro ring that unify $P_A(\xi)$ for all roots of unity for other $A\in \pi_0(\mathrm{Hom}(\pi_1(M),SL(2,\C)))$ as well (cf. \cite{Garoufalidis:2021lcp,Wheeler:2023cht} for the hyperbolic case).

    \item When the geometry is hyperbolic (which, in a sense, is the generic case), we have $N=\mathbb{H}^3$, the map $\mathrm{Isom}_+(N)\rightarrow PSL(2,\C)$ is the canonical isomorphism and $A_*$ is the standard geometric flat connection for the hyperbolic $M$. In this case, $i\CS[A_*]$ has the largest real part among all $i\CS[A]$ and thus the corresponding term in (\ref{intro-conj-exp}) exponentially dominates all other terms. Disregarding the other terms recovers the quantum modularity conjecture for hyperbolic manifolds by Wheeler \cite{Wheeler:2023cht}. For other Thurston geometries, however, $i\CS[A]$ is real for all flat connections and therefore none of the terms are exponentially suppressed. 

    \item In (\ref{intro-conj-exp}), $\CS[A]$ denotes a lift of the Chern--Simons value to $\C$ from $\C/\Z$. A change of the lift is equivalent to the change of $P_A(\tilde\xi)$ by an overall integer power of $\tilde\xi$. In the spherical case, due to the non-trivial shift by $C_N$ the relation (\ref{zhs-P-W}) assumes the specific lift given by (\ref{CS-geom-spherical}).

    \item The expansion (\ref{intro-conj-exp}) is consistent with Witten's asymptotic expansion conjecture \cite{Witten:1988hf,Freed:1991wd,andersen2004asymptotic,andersen2025proofwittensasymptoticexpansion} which is of similar form, but is formulated for the special roots of unity $\xi=e^{\frac{2\pi i}{r}}$ and have the sum over $\pi_0(\mathrm{Hom}(\pi_1(M),SU(2)))$. Namely, treating $SU(2)$ as a subgroup of $SL(2,\C)$, the Witten's conjecture is recovered assuming that (cf. \cite{Wheeler:2023cht})
\begin{multline}
     P_A(1)=0, \\
     \forall A\in \pi_0(\mathrm{Hom}(\pi_1(M),SL(2,\C)))\setminus \pi_0(\mathrm{Hom}(\pi_1(M),SU(2))),\; \Im\CS[A]\leq 0.
\end{multline}
Moreover, similarly to the case of the standard Witten's asymptotic expansion conjecture, (\ref{intro-conj-exp}), can be formally interpreted as the saddle point expansion of a certain ``contour path integral'' (in the sense of \cite{Witten:2010cx,Kontsevich}) in the infinite-dimensional space of all $SL(2,\C)$ connections on $M$, modulo the subgroup of gauge transformations $\{M\rightarrow SL(2,\C)\}$ with the degree being a multiple of $s$. The degree here can be defined as the coefficient of proportionality between the image of $[M]\in H_3(M)$ and $[SU(2)]\in H_3(SL(2,\C))$. This makes $e^{2\pi i\,\frac{r}{s}\CS}$ a well-defined function on this quotient space. This way $W_M(\xi)$ can be interpreted as the partition function of analytically continued $SU(2)$ Chern--Simons gauge theory for the rational level $r/s$. The quotient space can be equivalently understood as the $s$-fold cover of the standard space of connections modulo all gauge transformations.  The $s$ integer coefficients of the polynomials $P_A(\tilde{\xi})=\sum_{m=0}^{\lvert s \rvert -1} n_{A,m}^{(s)}\tilde{\xi}^{m}$ (defined for $\tilde{\xi}^s=1$) can be understood as coefficients of the decomposition of the contour into Lefschetz thimble contours corresponding to the $s$ lifts of the connected components of the $SL(2,\C)$ flat connections (under the assumption that a single thimble corresponds to a single copy of a connected component). Note that the TQFTs extending the WRT invariant for roots of unity others than $\xi=e^{\pm\frac{2\pi i}{r}}$ are known to be non-unitary. In the context of the ``contour path integral'' realization, the non-unitarity appears through the fact that although all the coefficients of the Chern--Simons action are real, the contour is not invariant under complex conjugation for $s\neq \pm 1$.
   
\item The asymptotic expansion statement (\ref{intro-conj-exp}) can be upgraded into an exact equality by using Borel resummation of the formal power series $I_A(s/r)$ (cf. \cite{garoufalidis2007chern,Gukov:2016njj,Garoufalidis:2020nut,andersen2022resurgence,Wheeler:2023cht,Gukov:2024vbr}). This statement is motivated by the ``contour path integral'' interpretation in the previous remark, because in the finite-dimensional case the Borel resummation is known to recover the integral over the Lefschetz thimble contour. Such a stronger version of the conjecture will not be the focus of this paper.

\item The assumption that $M$ is an integer homology sphere restricts the model geometry $N$ to three options: $\mathbb{H}^3$, $\widetilde{SL(2,\R)}$, or $S^3$. In the main text, we also present a more general (but technically much more involved) version for rational homology spheres, which allows other geometries. 

\item As in \cite{Wheeler:2023cht} for the hyperbolic case, an analogue of this conjecture could also be made for the so-called $\hat{Z}$-invariant of \cite{Gukov:2016gkn,Gukov:2017kmk}. This invariant is valued in $q$-series with integer coefficients. Namely, in the left-hand side of (\ref{intro-conj-exp}) $W(q)$ would be replaced with $\hat{Z}(q)$ evaluated at $q=e^{2\pi i\tau},\;\Im \tau >0$. The analogues of the factors $P_A(\xi)$'s in the right-hand side then would be power series in $\tilde{q} \coloneqq e^{-\frac{2\pi i}{\tau}}$ with integer coefficients. One of them, for the same special flat connection $A_*$ as before, would then be conjecturally equal to $\hat{Z}(\tilde{q})$. The formal series $I_A(s/r)$ would be replaced with a series in $\tau$ (the same as before, up to a simple normalization related factor). A weaker version of such a conjecture for $\hat{Z}$, formulated for possibly non-geometric integer homology spheres, but without a geometric interpretation of the special flat connection that recovers $\hat{Z}$, appeared in \cite{Gukov:2024vbr}. There it was also noticed that for some 3-manifolds the sum over $\pi_0(\mathrm{Hom}(\pi_1(M),SL(2, \C)))$ needs to be extended to a larger set, which includes also critical points of the Chern--Simons functional at infinity of the space of $SL(2,\C)$ connections, but with the finite value of the functional.

\end{itemize}

The rest of the paper is organized as follows. In Section~\ref{sec:review}, we present some basic facts about Thurston's geometrization, Seifert fibration, Chern--Simons functional, and quantum invariants of 3-manifolds, which will be relevant for the rest of the paper. We also introduce the notion of the geometric flat connection there. Section~\ref{sec:q-modularity} contains a more detailed formulation of the conjecture above, as well as its more general version for rational homology spheres. We prove that the conjecture holds for all Brieskorn homology spheres (Theorem~\ref{thm:bri}) and some other examples. Appendices contain technical details that are used in the main text.

\section{Geometrization and quantum invariants} \label{sec:review}

\subsection{Thurston's geometrization}

As was already mentioned in Section~\ref{sec:intro}, according to Thurston's geometrization, any 3-manifold can be decomposed into geometric pieces. Throughout the paper, we focus on closed oriented geometric 3-manifolds. For such manifolds the decomposition consists of a single component and thus $M\cong N/\Gamma$ where $N$ is one of the eight simply-connected model geometries listed in the first column of Table~\ref{tab:geometries} and $\Gamma \subset \mathrm{Isom}_+(N)$ is a discrete subgroup of the group of orientation-preserving isometries. As usual, $\mathbb{H}^n$, $S^n$, and $\R^n$ denote $n$-dimensional hyperbolic space, sphere, and Euclidean space, respectively, with the standard metric. The geometries $\widetilde{SL(2,\R)}$, $\Nil_3$, a $\Sol_3$ are Lie groups equipped with left-invariant metrics. The group $\widetilde{SL(2,\R)}$ is the universal cover of $SL(2,\R)$. Geometrically, it can be understood as an $\R$-fibration over $\mathbb{H}^2$ (topologically trivial, but not geometrically). The group $\Nil_3$ is the Heisenberg group, that is, the group of upper-triangular 3-by-3 matrices with 1's on the diagonal. Geometrically, it can be understood as an $\R$-fibration over $\R^2$. Finally, the group $\Sol_3$ can be realized as the semidirect product $\R\rtimes \R^2$ with $t\in \R$ acting on $\R^2$ as $t:(x,y)\mapsto (e^{-t}x,e^ty)$.

\begin{table}[h]
    \centering
    \begin{tabular}{c|cc} 
    \toprule
    $N$ & $\mathrm{Isom}_0(N)$ & $\longrightarrow PSL(2,\C)$ \\
    \midrule
    $\mathbb{H}^3$          & $PSL(2,\C)$ & $\mathrm{id}$ \\
    $S^3$                   & $SO(4)\cong P(SU(2)\times SU(2))$ & $\rightarrow PSU(2)\hookrightarrow PSL(2,\C)$ \\
    $\R^3$                  & $\R^3\rtimes SO(3)$ & $\rightarrow SO(3)\cong PSU(2)\hookrightarrow PSL(2,\C)$ \\
    $S^2\times \R$          & $SO(3) \times \R$ & $\rightarrow SO(3)\cong PSU(2)\hookrightarrow PSL(2,\C)$ \\
    $\mathbb{H}^2\times \R$ & $PSL(2,\R) \times \R$ &  $\rightarrow PSL(2,\R)\hookrightarrow PSL(2,\C)$ \\
    $\widetilde{SL(2,\R)}$  & $\R\rightarrow\mathrm{Isom}_0(N)\rightarrow PSL(2,\R)$ & $\hookrightarrow PSL(2,\C)$  \\
    $\Nil_3$                & $\R\rightarrow\mathrm{Isom}_0(N)\rightarrow \R^2\rtimes SO(2)$ & $\cong \C\rtimes U(1)\hookrightarrow PSL(2,\C)$  \\
    $\Sol_3$                & $\Sol_3$ & $\rightarrow \R \hookrightarrow PSL(2,\C)$   \\
    \bottomrule
    \end{tabular}
    \caption{The eight Thurston geometries, the connected components of the identity of their isometry groups, and natural maps to $PSL(2,\C)$.}
    \label{tab:geometries}
\end{table}

We will further assume that $\Gamma\subset \mathrm{Isom}_0(N)\subseteq \mathrm{Isom}_+(N)$ is inside the connected component of the identity. These groups are indicated in the second column Table~\ref{tab:geometries}. In the case of $\widetilde{SL(2,\R)}$ and $\Nil_3$ geometries, a short exact sequence into which $\mathrm{Isom}_0(N)$ fits is shown. In the list, the special orthogonal groups $SO(n)$ appear as the orientation-preserving isometry groups of spheres, $PSL(2,\R)$ appear as the orientation-preserving isometries of $\mathbb{H}^2$, and the abelian groups $\R^n$ appear as the groups of translations in Euclidean spaces. We refer to \cite{scott1983geometries} for the details.

\subsubsection{Geometric flat connection}
\label{sec:geom-flat}

The last column of Table~\ref{tab:geometries} shows a natural map $\mathrm{Isom}_0(N)\rightarrow PSL(2,\C)$ which is obtained by the composition of the natural projection and inclusion. The $S^3$ case is special, as one can choose to project on either of the two factors in the projective product. Due to this, in the case of the spherical geometry, we further assume that the subgroup $\Gamma\subset SO(4)$ is contained in the factor on which the projection is performed, in the sense that the composition of the inclusion and projection on the other factor would be trivial. This effectively restricts $M=S^3/\Gamma$ to be a link of ADE type singularity $\C^2/\Gamma$, up to a change of the orientation. Note that spherical geometric manifolds are also exceptional in the sense that for all other 7 cases the fundamental group $\pi_1(M)\cong \Gamma$ uniquely determines the homeomorphism class of the 3-manifold. Imposing this additional assumption restores the uniqueness in the spherical case.

For $N=\mathbb{H}^3$, $S^3$ and $\R^3$ we have $\mathrm{Isom}_0(N)=\mathrm{Isom}_+(N)$. For $N=
\mathbb{H}^2\times \R$, $\widetilde{SL(2,\R)}$, and $\Sol_3$, even though $\mathrm{Isom}_0(N)\subsetneq \mathrm{Isom}_+(N)$, the condition that $\Gamma\subset \mathrm{Isom}_0(N)$ can be also relaxed to $\Gamma\subset \mathrm{Isom}_+(N)$. In these three cases the full $\mathrm{Isom}_+(N)$ has natural projection to $O^+(2,1)\cong \Z/2\Z\ltimes PSL(2,\R)$ or $\Z/2\Z \ltimes \R$ which can still be naturally embedded in $PSL(2,\C)\cong SO^+(3,1)$. This is not the case for $N=S^2\times \R$ and $\Nil_3$ cases, because there is no canonical map from $O(3)$ and $O(2)\ltimes \R^2$ to $SO^+(3,1)$.

Composing the isomorphism $\pi_1(M)\cong \Gamma\subset \mathrm{Isom}_0(N)$ with the map $\mathrm{Isom}_0(N)\rightarrow PSL(2,\C)$ thus gives us a map $\pi_1(M)\rightarrow PSL(2,\C)$ uniquely defined up to conjugation. The obstruction of lifting a $\alpha:\pi_1(M)\rightarrow PSL(2,\C)$ map to a $\tilde{\alpha}:\pi_1(M)\SL(2,\C)$ one with respect to the central extension $\Z/2\Z\rightarrow SL(2,\C)\rightarrow PSL(2,\C)$ is the analog of the Stiefel--Whitney class $w_2(\alpha)\in H^2(M,\Z/2)$ corresponding to the homotopy class of the map $M\rightarrow B^2\Z/2\Z$ induced by the fibration sequence $B\Z/2\Z\rightarrow BSL(2,\C)\rightarrow BPSL(2,\C)\rightarrow B^2\Z/2\Z$. For the map $\alpha:\pi_1(M)\rightarrow PSL(2,\C)$ defined by the geometric structure as above, we have $w_2(\alpha)=w_2(TM)$. Since any oriented 3-manifold is spin we have $w_2(TM)=0$ and therefore a lift $\tilde{\alpha}:\pi_1(M)\rightarrow SL(2,\C)$ exists. The lifts naturally correspond to spin structures. We will restrict ourselves to mod-2 homology spheres, that is, 3-manifolds with $H^i(M;\Z/2\Z)=0,\;i=1,2$, so that the lift is unique. This map will represent a particular connected component in the space $\mathrm{Hom}(\pi_1(M),SL(2,\C))/SL(2,\C)$ of $SL(2,\C)$ flat connections on $M,$ which will play a special role in the quantum modularity. We will refer to this flat connection as \textit{geometric}.

\subsubsection{Seifert manifolds}

\label{sec:seifert}

A Seifert manifold with $m$ exceptional fibers is a closed 3-manifold built out
of the following set of data: an $S^1$ fibration over a surface $S$ with Euler
number $b$, and $m$ pairs of coprime integers $(p_1, q_1)$, \ldots,
$(p_m, q_m)$.
The pairs of integers describe the $m$ exceptional fibers of the fibration. The
boundary of a tubular neighborhood of each fiber is a 2-torus with basis
1-cycles in homology given by the boundary of the disk, and a fiber at the
boundary. The tubular neighborhood of the $j$-th exceptional fiber is a solid
torus whose $(p_j, q_j)$ boundary cycle is contractible. Notice that the pair
$(p_j, q_j)$ is determined only up to an overall sign, so we can always choose
$p_j > 0$ and denote the Seifert manifold as manifold $S(b;\, p_1/q_1,\ldots,
p_m/q_m)$.
It is a fact that Seifert manifolds $S(b;\, p_1/q_1, \ldots, p_m/q_m)$ and
$S'(b';\, p'_1/q'_1, \ldots, p'_m/q'_m)$ are orientation preserving
homeomorphic if and only if $S = S'$, $e = e'$, up to reordering $p_j = p'_j$
and $q_j = q'_j \bmod p_j$ for every $j$ \cite[\S~3]{scott1983geometries}. In
light of this result, Seifert invariants are sometimes normalized so that $b = 0$.

Two important invariants of a Seifert fibration are its Euler number
\begin{equation}
    e = -b + \sum_{j = 1}^m \frac{q_j}{p_j}\;\in \Q,
\end{equation}
and the orbifold Euler characteristic of the base $M/S^1$,
\begin{equation}
    \chi = \chi(S) - \sum_{j = 1}^m \left(1 - \frac{1}{p_j}\right) \;\in \Q.
\end{equation}
For six of the eight Thurston geometries, all closed manifolds are Seifert
fibrations. Conversely, given a Seifert manifold, which model geometry it
admits is completely determined by the two invariants $e$ and $\chi$
according to the following table \cite[Table~4.1]{scott1983geometries}.

\begin{table}[h]
    \centering
    \begin{tabular}{c|ccc}
        \toprule
                   & $\chi > 0$      & $\chi = 0$ & $\chi < 0$            \\
        \midrule
        $e = 0$    & $S^2 \times \R$ & $\R^3$      & $\mathbb{H}^2 \times \R$       \\
        $e \neq 0$ & $S^3$           & $\Nil_3$        & $\widetilde{SL(2,\R)}$  \\
        \bottomrule
    \end{tabular}
\end{table}
In this paper, we will only consider Seifert fibrations over oriented bases. This, in particular, implies that the condition $\Gamma\subset \mathrm{Isom}_0(N)$ holds.

\begin{figure}[h]
     \centering
     \begin{tikzpicture}[scale=0.4]
         \begin{knot}[end tolerance=1pt]
         \strand[thick] (-6, 0)
           .. controls ++(90:1.8) and ++(0:-1) .. (0,2)
           .. controls ++(0:1) and ++(90:1.8) ..
           (6,0)
           .. controls ++(90:-1.8) and ++(0:1) .. (0,-2) node[below=3] {$b$}
           .. controls ++(0:-1) and ++(90:-1.8) .. (-6, 0);

         \strand[thick] (-7.5, 0)
           .. controls ++(90:-1.2) and ++(0:-1) .. (-6,-2.5)
           .. controls ++(0:1) and ++(270:1.2) .. (-4.5,0)
           .. controls ++(270:-1.2) and ++(0:1) .. (-6,2.5) node[above
left] {${p_1}/{q_1}$}
           .. controls ++(180:1) and ++(90:1.2) .. (-7.5,0)
         ;

         \strand[thick] (-7.5+4, 0+2)
           .. controls ++(90:-1.2) and ++(0:-1) .. (-6+4,-2.5+2)
           .. controls ++(0:1) and ++(270:1.2) .. (-4.5+4,0+2)
           .. controls ++(270:-1.2) and ++(0:1) .. (-6+4,2.5+2)
node[above] {${p_2}/{q_2}$}
           .. controls ++(180:1) and ++(90:1.2) .. (-7.5+4,0+2)
         ;

         \strand[thick] (-7.5+12, 0)
           .. controls ++(90:-1.2) and ++(0:-1) .. (-6+12,-2.5)
           .. controls ++(0:1) and ++(270:1.2) .. (-4.5+12,0)
           .. controls ++(270:-1.2) and ++(0:1) .. (-6+12,2.5)
node[above right] {${p_m}/{q_m}$}
           .. controls ++(180:1) and ++(90:1.2) .. (-7.5+12,0)
         ;

         \draw[thick, dotted, dash pattern= on 1.5pt off 5.7pt]
(0.7,3.5) arc (90:72:11);

         \flipcrossings{1,4,5}
         \end{knot}
     \end{tikzpicture}
     \caption{Surgery presentation of the manifold $S^2(b;\,p_1/q_1,
     \ldots,p_m/q_m)$.} 
     \label{fig:seifert-link}
\end{figure}
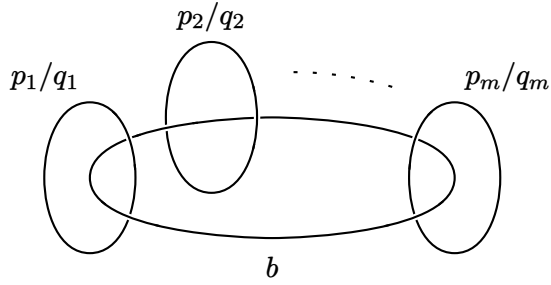

When the Seifert manifold is fibered over the $S^2$, it has an alternate
description as Dehn surgery on the link in Figure~\ref{fig:seifert-link} in
$S^3$. Furthermore, the first integral homology group of $M = S^2(b;\,
p_1/q_1,\ldots, p_m/q_m)$ is isomorphic to the cokernel of the matrix
\begin{equation}
    B = 
    \begin{bmatrix}
        b       & 1      & \cdots & 1       & 1      \\
        q_1     & p_1    & \cdots & 0       & 0      \\
        \vdots  & \vdots & \ddots & \vdots  & \vdots \\
        q_{m-1} & 0      & \cdots & p_{m-1} & 0      \\
        q_m     & 0      & \cdots & 0       & p_m
    \end{bmatrix},
\end{equation}
i.e., $H_1(M, \Z) \cong \Z^{m + 1} / B \Z^{m + 1}$
\cite[\S~1.1.4]{saveliev2002invariants}. Since $\det B = e p_1 p_2 \cdots p_m$,
when $M$ is a rational homology sphere we have $\lvert H_1(M, \Z) \rvert =
\lvert e p_1 p_2 \cdots p_m \rvert$. In particular, when $M$ is an integer
homology sphere, 
\begin{equation}
    p_1 p_2 \cdots p_m \left(-b + \sum_{j = 1}^m \frac{q_j}{p_j}\right) = \pm
    1
\end{equation}
which implies that (\textit{i}) $p_j$'s are mutually coprime (\textit{ii}) $q_j$'s are
uniquely determined modulo $p_j$ for every $j$, and therefore (\textit{iii}) $b$ is also
determined. Conversely, given $m$ mutually coprime integers $p_1$, \ldots,
$p_m$ such that $p_j \geq 2$, there is a unique Seifert integer homology
sphere with $m$ exceptional fibers, up to orientation reversal
\cite[\S~1.1.4]{saveliev2002invariants}.  We shall pick an orientation such that
$e > 0$ and denote this manifold as $\Sigma(p_1, \ldots, p_m)$.

We end this section by noting that with the exception of Poincar\'e homology
sphere, $\Sigma(2, 3, 5)$, for which $\chi > 0$, all other Seifert integer
homology spheres $\Sigma(p_1, \ldots, p_m)$ have $\chi < 0$ and therefore have
geometry modeled on $\widetilde{SL(2, \R)}$.

\subsection{Chern--Simons functional and flat connections}

An $SL(2,\C)$ principal bundle over a 3-manifold $M$ is always trivial. Choosing a global trivialization, we can consider the connection 1-form $A\in \Omega^1(M)\otimes \sl_2(\C)$. One then defines the Chern--Simons functional as follows:
\begin{equation}
    \CS[A]=\frac{1}{8\pi^2}\int_M\mathrm{Tr}\,\left(AdA+\frac{2}{3}A^3\right)\in \C.
\end{equation}
Under a change of the trivialization $g:M\rightarrow SL(2,\C)$ (i.e. a gauge transformation) we have $A'=gAg^{-1}+gdg^{-1}$ and the functional changes by an integer
\begin{equation}
    \CS[A']=\CS[A]+\mathrm{deg}\,g
    \label{CS-gauge-transformation}
\end{equation}
where $\mathrm{deg}\,g$ is defined as the coefficient of proportionality between $g_*[M]$ and the chosen generator of $H_3(SL(2,\C))\cong \Z$. Thus we can consider instead
\begin{equation}
    \CS[A]=\frac{1}{8\pi^2}\int_M\mathrm{Tr}\,\left(AdA+\frac{2}{3}A^3\right) \mod 1 \in \C/\Z
\end{equation}
which is independent of the choice of the trivialization. The critical points of the functional are connection 1-forms satisfying the flatness condition $F_A\equiv dA+A^2=0$. Modulo gauge transformations they form the moduli space of flat connections
\begin{equation}
    \mathcal{M}_\text{flat}(M,SL(2,\C))=\mathrm{Hom}(\pi_1 M,SL(2,\C))/SL(2,\C)
\end{equation}
where the action is by conjugation\footnote{In the context of asymptotic expansion of the quantum invariants, it is actually more natural to consider only semi-stable flat connections \cite{Witten:2010cx}, which would form a generally smaller space
\begin{equation}
    \mathcal{M}_\text{flat}^\text{s.s.}(M,SL(2,\C))=\mathrm{Hom}(\pi_1 M,SL(2,\C))\sslash SL(2,\C)
\end{equation}
given by the GIT quotient. On the other hand, one may also consider a larger moduli space $\mathcal{M}_\text{flat}^\text{s.s.+inf}(M,SL(2,\C))$ that also includes gauge equivalence classes of critical points of the Chern--Simons functional at infinity with finite critical values \cite{Gukov:2024vbr}. One can expect that the most general quantum modularity conjecture should be formulated with such a modified moduli space. However, it will not be relevant for the class of 3-manifolds considered in this paper.
}. The Chern--Simons functional then can be reduced to a function on the finite set of connected components:
\begin{equation}
    \CS:\;\pi_0(\mathcal{M}_\text{flat}(M,SL(2,\C)))\longrightarrow \C/\Z.
\end{equation}
Depending on the context, we will also use $\CS$ to denote a lift of the value to $\C$. Such a lift in particular can be provided by a specific choice of the representative connection 1-form.

The moduli space contains a distinguished subspace of abelian flat connections
\begin{multline}
    \mathcal{M}_\text{flat}(M,SL(2,\C)) \supset \mathcal{M}_\text{flat}^{\text{ab}}(M,SL(2,\C))\\
    =\mathrm{Hom}(\pi_1 M,\C^*)/\{\pm 1\}\cong H^1(M,\C/\Z)/\{\pm 1\}
\end{multline}
corresponding to the maps $\pi_1M\rightarrow SL(2,\C)$ that can be conjugated into a maximal torus subgroup $\C^*\subset SL(2,\C)$. The $\{\pm 1\}$ is the Weyl group of $\sl_2$, the generator of which acts as the inversion automorphism of $\C^*\cong \C/\Z$. We have a canonical identification 
\begin{equation}
    \pi_0(\mathcal{M}_\text{flat}^{\text{ab}}(M,SL(2,\C)))=\mathrm{Hom}(\mathrm{Tor}\, H_1( M),\Q/\Z)/\{\pm 1\}\cong \mathrm{Tor}\,H_1(M)/\{\pm 1\}
\end{equation}
where the last bijection is provided by the linking pairing
\begin{equation}
    \mathrm{lk}:\;\mathrm{Tor}\,H_1(M)\otimes_\Z \mathrm{Tor}\,H_1(M)\longrightarrow \Q/\Z.
\end{equation}
The value of the Chern--Simons functional restricted to this subset of connected components of the moduli space is then given  by
\begin{equation}
    \begin{array}{rrcl}
        \CS: & \pi_0(\mathcal{M}_\text{flat}^{\text{ab}}(M,SL(2,\C)))\cong \mathrm{Tor}\,H_1(M)/\{\pm 1\} & \longrightarrow  & \C/\Z, \\
         & a & \longmapsto & \mathrm{lk}(a,a).
    \end{array}
\end{equation}
It will be also useful to consider the subspace of non-abelian flat connections:
\begin{equation}
    \mathcal M^\text{non-ab}_\text{flat}(M,SL(2, \C))) \coloneqq \mathcal M_\text{flat}(M,SL(2, \C)))\setminus \mathcal M^\text{ab}_\text{flat}(M,SL(2, \C))).
\end{equation}

Consider now the case of the geometric $SL(2,\C)$ flat connection in the case of spherical geometry $N=S^3$, as defined in Section~\ref{sec:geom-flat}. The universal cover of $M=S^3/\Gamma$, is $S^3$, which we can identify with the $SU(2)$ group, with the identity element being the base-point. Under our assumptions $\Gamma\subset SU(2)\subset \mathrm{Isom}_0(S^3)=SO(4)$. It acts on $S^3\cong SU(2)$ by left or right multiplication, depending on the choice of the $SU(2)$ subgroup in $SO(4)$ containing $\Gamma$. One can choose a particular representative connection 1-form $A_*\in \Omega^1(M)\otimes \mathfrak{su}_2$ (with $\mathfrak{su}_2\subset \mathfrak{sl}_2(\C)$) such that its lift $\tilde{A}_*\in \Omega^1(S^3)\otimes \mathfrak{su}_2$ to the universal cover is the pure gauge: $\tilde{A}_*=gdg^{-1}$ for the identity map $g:S^3\cong SU(2)\to SU(2)$. Thus, from (\ref{CS-gauge-transformation}) we have $\CS[\tilde{A}_*]=\pm 1$, with the sign depending on the choice of orientation.  On the other hand, $\CS[\tilde{A}_*]=\lvert \Gamma \rvert \CS[A_*]$, because the integral of the CS 3-form over $S^3$ is just the sum of $\lvert \Gamma \rvert$ copies of the integral over a fundamental region, which is the same as $\CS[A_*]$. Thus we have
\begin{equation}
    \CS[A_*]=\pm\frac{1}{\lvert \Gamma \rvert}
    \label{CS-geom-spherical}
\end{equation}
for the geometric flat connection in the spherical case.

In the case of $N=\widetilde{SL(2,\R)}$ geometry, the value of the Chern--Simons functional for the geometric flat connection is known to be
\begin{equation}
    \CS[A_*]=\pm\frac{\chi^2}{4e}
    \label{CS-geom-sl2r}
\end{equation}
for the canonical representative 1-form $A_*\in \Omega^1(M)\otimes \sl_2(\R)$, where $\chi$ and $e$ are the invariants of the Seifert fibration \cite{brooks1984godbillon,brooks1984volumes,khoi2003cut}. In this case, possibly up to a sign, it is known to coincide with the so-called Seifert volume and Godbillon--Vey invariant. As in the spherical case, the sign depends on the choice of the orientation. 

Remarkably, $4e/\chi^2= \lvert \Gamma \rvert$ for the spherical manifolds with $\Gamma\subset SU(2)\subset SO(4)$ (which have ADE classification), which can be verified by a direct calculation. Therefore, one can use the expression (\ref{CS-geom-sl2r}) for the geometric flat connection in both cases.

\subsection{Review of WRT invariants at generic roots of unity}

The construction of Reshetikhin and Turaev \cite{RT91,turaev1992modular} provides a numerical invariant of closed oriented 3-manifolds via their Dehn surgery representation. The invariant is defined for any modular tensor category and can also be extended to a 3-2-1 TQFT. When the input is (the semi-simplification of the) category of finite-dimensional representations of the Hopf
algebra $\bar U_\xi \sl_2$ for a root of unity $\xi$, the invariant is known as the $\sl_2$ Witten--Reshetikhin--Turaev invariant. Moreover, when $\xi=e^{\frac{2\pi i}{r}}$, this gives a mathematical meaning to
the Chern--Simons path integral considered by Witten \cite{Witten:1988hf}. We will consider the standard, also known as $SU(2)$, version of the invariant. For mod-2 homology spheres, it contains essentially the same information as the $SO(3)$ invariant \cite{KM91}. For general 3-manifolds, however, one may expect that the quantum modularity is most naturally formulated as the relation between $SU(2)$ and $SO(3)$ invariants, as these groups are Langlands dual to each other.

The $\sl_2$ invariant can also be defined directly in terms of colored Jones polynomials. Namely let $L$ be an $N$-component framed link in $S^3$ with components
$\{L_1, \ldots, L_N\}$ colored by irreducible $\sl_2$ representations of dimensions $n_1,n_2,\ldots,n_N$. Then, using either skein relations together with the cabling operations, or the ribbon structure on the category of finite-dimensional representations of the quantum group $U_q(\sl_2)$ for generic $q$, one can define the colored Jones polynomial invariant of the link
\begin{equation}
    J_n(L;q)\in q^{a/4}\Z[q^{\pm 1}],\qquad a\in\Z,
\end{equation}
where $n \coloneqq (n_1,\ldots,n_m)\in \Z_{\geq 1}^m$. We consider normalization such that for the unknot $U$,
\begin{equation}
    J_n(U;q)=[n]_q \coloneqq \frac{q^{n/2}-q^{-n/2}}{q^{1/2}-q^{-1/2}}.
\end{equation}

Let $M = S^3(L)$ be the 3-manifold obtained by surgery on $L$. Fix $\xi^{1/4}$
to be a primitive $4r$-th root of unity. To avoid some technical subtleties
later on, we will assume that $r$ is odd and $s = 1 \bmod 4$. With the intermediate
quantity,
\begin{equation}
    F(L, \xi) \coloneqq \sum_{n \in \{1, \ldots, r - 1\}^m} J_n(L, \xi) \prod_{L_i
    \subset L} [n_i]_\xi,
\end{equation}
the WRT invariant can be written as
\begin{equation}
    \tau_M(\xi) \coloneqq \frac{F(L, \xi)}{F(U^{+1}, \xi)^{b_+} F(U^{-1},
    \xi)^{b_-}},
    \label{eqn:wrt-def}
\end{equation}
where $U^{\pm 1}$ is understood to be a $\pm 1$-framed unknot, and $b_{\pm}$
is the number of positive/negative eigenvalues of the linking matrix of $L$. As
defined above, the invariant is normalized to be $\tau_{S^3}(\xi) = 1$. Although in general the invariant depends on the choice of $\xi^{1/4}$, for integer homology spheres it actually depends only on $\xi$ itself \cite{habiro2002quantum}.

\section{Quantum modularity}
\label{sec:q-modularity}

To state the quantum modularity conjecture, it will be convenient to normalize
the WRT invariant in the following way; suppose that $M$ is a rational homology
sphere, let $H \coloneqq \lvert H_1(M, \Z) \rvert$ be the order of its first
homology group, and define
\begin{equation}
    W_M(\xi) \coloneqq \sqrt{H} \left(\frac{H}{s}\right) (\xi - 1)
    \tau_M(\xi).
\end{equation}
Since $s$ is an odd integer, the Jacobi symbol $\left(H/s\right)$ is
well-defined.

When $M$ is an integer homology sphere, $W_M(\xi) = (\xi - 1) \tau_M(\xi)$
and the statement takes the form of Conjecture~\ref{conj:zhs}. 
We propose the following generalization for rational homology spheres, which are also mod-2 homology spheres.
\begin{conjecture}
    \label{conj:qhs}
    Let $M\cong N/\pi_1(M)$ be a geometric rational homology sphere with no 2-torsion in homology and $\xi^{1/4}$ be a $4r$-th primitive root of unity. Specifically, take $\xi^{1/4}=e^{\frac{\pi is}{2r}}$ for some $r\in 2\Z_{\geq 1}+1$ and $s\in 4\Z+1$. Moreover, assume that $s$ is coprime with the order of the
    homology group $H = \lvert H_1(M, \Z) \rvert$. Then 
    \begin{enumerate}
        \item the WRT invariant admits the decomposition into a sum of terms labeled by abelian flat connections
            \begin{equation} W_M(\xi) = \sum_{a \in \pi_0(\mathcal{M}_\text{flat}^\text{ab}(M,
                SL(2,\C)))}  e^{2\pi i \frac{r}{s} \CS[a]}\,W^{a}_M (\xi);
                \label{conj-rhs-ab-decomp}
            \end{equation}
        such that each term has the
            following asymptotic expansion as $r \to \infty$ along integers coprime with $s$ and $H$:
            \begin{equation}
                e^{2\pi i \frac{r}{s} \CS[a]} W^{a}_M(\xi) \simeq \sum_{A \in \pi_0(\mathcal M_\text{flat}(M,
                SL(2, \C)))} e^{2\pi i \frac{r}{s} \CS[A]} P_A^{a} (\tilde \xi)
                 I^{a}_A(s/r),
            \end{equation}
            where $\tilde\xi = e^{-\frac{2\pi i r}{s}}$, $P_A^{a}(\tilde{\xi})$
            are polynomials in $\tilde\xi$ (with $s$-dependent coefficients),
            and $I^{a}_A(s/r)\in (s/r)^{\delta_A/2}\C \llbracket s/r
            \rrbracket$ for some $\delta_A\in \Z$. Moreover, $P_A\ap{triv}(\tilde \xi)\in\Z[\tilde{\xi}]$ and
               \begin{equation}
            P^{a}_{A}(\tilde\xi)=0
        \end{equation}
            for $A\in \pi_0(\mathcal{M}^\text{ab}_\text{flat}(M,SL(2,\C))) \setminus \{a\}$.
        \item
        for the geometric flat connection $A_*$
            \begin{equation}
                P_{A_*}\ap{triv}(\xi) = \xi^\delta \sum_{a \in \pi_0(\mathcal{M}\ped{flat}(M,
                U(1)))} W^{a}_M(\xi) + C_M,
                \label{conj-rhs-PW}
            \end{equation}
            where $\delta\in \Q$ and 
            \begin{equation}
             C_N=\left\{
                \begin{array}{cc}
                     -1 & N=S^3, \\
                     0 & \text{otherwise.}
                \end{array}
             \right.
            \end{equation}
    \end{enumerate}
\end{conjecture}
We remark that the decomposition (\ref{conj-rhs-ab-decomp}) with respect to the abelian flat connections is motivated by the analogous decomposition of the WRT invariant at $\xi=e^{\frac{2\pi i}{r}}$ in its relation to the $\hat{Z}$-invariants \cite{Gukov:2016gkn,Gukov:2016njj,Gukov:2017kmk,Murakami:2023oam}. The sum on the right-hand side of \ref{conj-rhs-PW} is the analogue of $\hat{Z}_0$. In Appendix \ref{app:Zhat} we provide a conjectural formula that relates $\hat{Z}$-invariants to the WRT invariant at a general root of unity. It generalizes the relation  Such a relation allows one to draw a direct parallel between the quantum modularity considered in this paper and \textit{holomorphic} quantum modularity of $\hat{Z}$-invariants.

\subsection{Proof of conjecture for Brieskorn homology spheres}

In this section, we will prove the quantum modularity conjecture for Brieskorn
homology spheres, i.e., Seifert integer homology spheres with three exceptional
fibers. Such manifolds admit $\widetilde{SL(2,\R)}$ or $S^3$ geometry (Poincar\'e homology sphere). The proof we give here is similar in spirit to \cite{LZ99, Hikami05Brieskorn}
and proceeds by identifying WRT invariants as limits of Eichler integrals of
half-integral weight modular forms---also known as false theta functions---and using the
modular transformation properties of the Eichler integral to make a statement
about the asymptotics of the WRT invariant. We begin by recalling the surgery
formula for WRT invariants at generic roots of unity due to
\cite[\S~4.4]{LR99}, and modular transformation properties of a family of
weight-$3/2$ vector-valued modular forms from \cite{Hikami05Brieskorn}. We then
prove the (quantum) modular and integrality properties of the limits of their
Eichler integrals at arbitrary rational numbers, and show how these limits are
related to the WRT invariant at generic roots of unity. Finally, we use the
properties of the Eichler integral, and a correspondence between non-trivial
flat connections and components of the vector-valued modular form to prove the
quantum modularity conjecture.

\paragraph{WRT invariant of Seifert integer homology spheres}

Suppose that $M$ is a Seifert integer homology sphere $\Sigma(p_1, \ldots,
p_m)$ with $m$ exceptional fibers, where $p_j \geq 2$ are mutually coprime
integers. As stated Section~\ref{sec:seifert}, Seifert invariants $p_1/q_1$,
\ldots, $p_m/q_m$, and $b$ are determined from the requirement that the
manifold is an integer homology sphere. Furthermore, we shall also pick a
normalization in which $b = 0$ and the orientation such that $e > 0$. Let
\begin{equation}
    P \coloneqq \prod_{j = 1}^m p_j,
    \ 
    H \coloneqq P\sum_{j = 1}^m \frac{q_j}{p_j} = \lvert H_1(M, \Z)\rvert,
    \ \text{and}\ 
    \phi \coloneqq 3\sigma(H) + \sum_{j = 1}^m \left(12s(q_j,
    p_j) - \frac{q_j}{p_j}\right),
\end{equation}
where $s(\cdot, \cdot)$ is the Dedekind sum, then, with Seifert invariants normalized as above, the WRT invariant of $M$ at $\xi^{1/4} = e^{\pi i s/2r}$ has
the following form \cite{LR99}
\begin{equation}
    \xi^{\phi/4 - 1/2} (\xi - 1) \tau_M(\xi)
    = \left(\frac{Pr}{s}\right) \frac{e^{\pi i / 4}}{2\sqrt{2Pr}}
    \sum_{\substack{n \in \Z/2Pr\Z \\ r \nmid n}} \xi^{-Hn^2/4P} \frac{\prod_{j =
    1}^m (\xi^{n/2p_j} - n^{-n/2p_j})}{(\xi^{n/2} - \xi^{-n/2})^{m - 2}}.
    \label{eqn:wrt-seifert}
\end{equation}
We know from the integrality of the WRT invariant of integer homology spheres
that $\tau_M(\xi) \in \Z[\xi]$. The following proposition says that the
fractional part of the power of $\xi$ on the left-hand side of
(\ref{eqn:wrt-seifert}) is equal to the Chern--Simons action at the geometric
flat connection, $\CS[A_*] = -\chi^2 / 4e$.

\begin{proposition}
    \label{prop:phi-cs}
    Consider the Seifert homology sphere $\Sigma(p_1, p_2, \ldots,p_m)$, then
    with $\phi$ as defined above, we have $\phi/4 - 1/2 = -\chi^2/4e \bmod 1$.
\end{proposition}
\begin{proof}
    We begin by recalling how the quantity $\phi$ is related to the
    Casson invariant $\lambda$ \cite[\S 4]{LR99} 
    \begin{equation}
        -24\lambda = \phi + \frac{P}{H} \left(m - 2 - \sum_{j = 1}^m
        \frac{1}{p_j^2} \right),
    \end{equation}
    which implies that
    \begin{equation}
        \frac{\phi}{4} = -\frac{P}{4} \left(m - 2 - \sum_{j = 1}^m \frac{1}{p_j^2}\right)
        \mod 1.
    \end{equation}
    A short computation with the definition of the orbifold Euler
    characteristic, and the fact that $e = 1/P$ yields
    \begin{equation}
        \frac{\phi}{4} - \frac{1}{2} + \frac{\chi^2}{4e} = \frac{P}{2}
        \left(\frac{m(m + 1)}{2} - 1 - m \sum_j \frac{1}{p_j} + \sum_{j < k}
        \frac{1}{p_j p_k} \right) - \frac{1}{2} \mod 1.
    \end{equation}
    To show that the right-hand side is an integer (and therefore zero modulo
    1), it suffices to show that the integer
    \begin{equation}
        P \left(\frac{m(m + 1)}{2} - 1 - m \sum_j \frac{1}{p_j} + \sum_{j < k}
        \frac{1}{p_j p_k} \right)
        \label{eqn:odd}
    \end{equation}
    is odd. Note that since $p_j$'s have to be mutually coprime because
    $\Sigma(p_1, p_2,\ldots, p_m)$ is an integer homology sphere, at most one
    of them can be even. With this observation, we have
    \begin{equation}
        P \sum_j \frac{1}{p_j} = 
        \begin{cases}
            1  & \text{if $P$ is even} \\
            m  & \text{if $P$ is odd}
        \end{cases} \mod 2,
    \end{equation}
    and
    \begin{equation}
        P \sum_{j < k} \frac{1}{p_j p_k} = 
        \begin{cases}
            m - 1 & \text{if $P$ is even} \\
            \frac{m(m - 1)}{2} & \text{if $P$ is odd}
        \end{cases} \mod 2,
    \end{equation}
    which implies that (\ref{eqn:odd}) is odd.
\end{proof}

\paragraph{Modular forms and false theta functions}

Following \cite{Hikami05, Hikami05Brieskorn, Hikami06}, we describe
bases of weight-$3/2$ vector-valued modular forms and their Eichler integrals
whose limit at rational numbers are weight-$1/2$ strong quantum modular forms.

We specialize to $m = 3$, the case relevant for Brieskorn homology spheres.
Given a triple, $\vec{p} = (p_1, p_2, p_3)$, let $P = p_1 p_2 p_3$, and for
every triple, $\vec{a} = (a_1, a_3, a_3)$, with integers $a_j$ such that $0
< a_j < p_j$, define the $2P$-periodic functions
\begin{equation}
    \phi_{\vec p}^{\vec a} (l) = 
    \begin{cases}
        - \epsilon_1 \epsilon_2 \epsilon_3 & \text{if}\ l = P\left(1 + \sum_{j =
        1}^3 \frac{\epsilon_j a_j}{p_j}\right) \bmod 2P, \\
        0 & \text{otherwise},
    \end{cases}
\end{equation}
where $\epsilon_j = \pm 1$. However, not all $\phi^{\vec a}_{\vec p}$'s are
independent. Indeed, for $j = 1, 2, 3$ define the involutions
\begin{equation}
    \sigma_1(\vec a) = (p_1 - a_1, a_2, a_3),\ 
    \sigma_2(\vec a) = (a_1, p_2 - a_2, a_3),\ \text{and}\ 
    \sigma_3(\vec a) = (a_1, a_2, p_3 - a_3),
\end{equation}
and note that for $j \neq k$ we have $\phi^{\sigma_j \circ \sigma_k(\vec
a)}_{\vec p}(l) = \phi^{\vec a}_{\vec p} (l)$, so that the number of such
independent periodic functions is
\begin{equation}
    D_{\vec p} = \frac{1}{4} (p_1 - 1) (p_2 - 1) (p_3 - 1).
\end{equation}
For later use, we also note the following properties of the $\phi_{\vec
p}^{\vec a}$'s,
\begin{equation}
    \phi_{\vec p}^{\vec a}(-l) = -\phi_{\vec p}^{\vec a}(l)
    \quad\text{and}\quad
    \sum_{l = 0}^{2P - 1} \phi_{\vec p}^{\vec a}(l) = 0.
\end{equation}
For each independent $\phi^{\vec a}_{\vec p}$, define theta functions on the
upper half plane
\begin{equation}
    \Phi^{\vec a}_{\vec p}(\tau) = \frac{1}{2} \sum_{l \in Z} l \phi^{\vec
    a}_{\vec p}(l) q^{l^2/4P}
    \quad\text{with $q = e^{2\pi i \tau}$},
\end{equation}
which are components of a $D_{\vec p}$-dimensional vector-valued modular form
of weight $3/2$. Under $T$ and $S$ modular transformations, these theta
functions behave as \cite[Proposition~2]{Hikami05Brieskorn}
\begin{equation}
    \Phi^{\vec a}_{\vec p}(\tau + 1) = T^{\vec a} (\vec p) \Phi^{\vec a}_{\vec
    p}(\tau)
    \quad\text{and}\quad
    \Phi^{\vec a}_{\vec p}(\tau) = \left(\frac{i}{\tau}\right)^{3/2} \sum_{\vec
    b} S^{\vec a}_{\vec b} (\vec p) \Phi^{\vec b}_{\vec p}(-1/\tau)
\end{equation}
respectively, where
\begin{equation}
    T^{\vec a} = \exp\left(\frac{P}{2}\left(1 + \sum_{j = 1}^3
    \frac{a_j}{p_j}\right)^2 \pi i \right)
\end{equation}
and
\begin{equation}
    S_{\vec b}^{\vec a} = - \frac{8}{\sqrt{2P}} (-1)^{P \left(1 +
    \sum_j\frac{a_j + b_j}{p_j}\right) + P \sum_{j \neq k} \frac{a_j b_k}{p_j
    p_k}} \prod_j \sin\left(P \frac{a_j b_j}{p_j^2} \pi\right).
\end{equation}
We also introduce the so-called false theta functions, which are Eichler integrals
of the theta functions defined above
\begin{equation}
    \tilde\Phi_{\vec p}^{\vec a}(\tau) = \sum_{l = 0}^\infty \phi_{\vec p}^{\vec
    a}(l) q^{l^2 / 4P},
\end{equation}
convergent in the upper half plane $\Im \tau > 0$, whose limits at rational
numbers $\tau \to s/r$ follow from a proposition of \cite{LZ99} (see also
\cite[Proposition~3]{Hikami05Brieskorn})
\begin{equation}
    \tilde\Phi_{\vec p}^{\vec a}(s/r) \coloneqq \lim_{t \to 0^+}
    \tilde\Phi_{\vec p}^{\vec a}(s/r + it) = \frac{1}{2} \sum_{l = 0}^{2Pr} \phi_{\vec
    p}^{\vec a}(l) \xi^{l^2/4P} \left(1 - \frac{l}{Pr}\right).
\end{equation}
From the formula above, it is straightforward to deduce the behavior of these
functions under a modular $T$-transformation, $\tilde\Phi_{\vec p}^{\vec
a}(\tau + 1) = T^{\vec a}(\vec p) \tilde\Phi_{\vec p}^{\vec a}(\tau)$, where
$T^{\vec a}(\vec p)$. And while the false theta functions do not have good
transformation behavior under the action of the full modular group, their limits at
cusps, $\tau \to r/s$, satisfy the following quantum modular property under an
$S$-transformation as $r \to \infty$ \cite[Proposition~8]{Hikami05Brieskorn}
\begin{equation}
    \tilde\Phi_{\vec p}^{\vec a}(s/r) + \sqrt{\frac{r}{is}} \sum_{\vec b}
    S_{\vec b}^{\vec a}(\vec p) \tilde\Phi_{\vec p}^{\vec b}(-r/s) \simeq \sum_{k =
    0}^\infty \frac{L(-2k, \phi^{\vec a}_{\vec p})}{k!} \left(\frac{\pi i s}{2
    P r}\right)^k, 
\end{equation}
where the sum over $\vec b$ includes only the $D_{\vec p}$ independent false
theta functions, and the coefficients of the power series on the right-hand side are
\begin{equation}
    L(-2k, \phi^{\vec a}_{\vec p}) = - \frac{(2P)^{2k}}{2k + 1} \sum_{l = 1}^{2P}
    \phi^{\vec a}_{\vec p}(l) B_{2k + 1}(l/2P),
\end{equation}
where $B_{2k + 1}$ denotes the $(2k + 1)$-th Bernoulli polynomial, are values of
the analytically continued $L$-function for the Dirichlet character $\phi^{\vec
a}_{\vec p}$.

\paragraph{False theta functions and flat connections}

Here, we review a characterization of the connected components of the moduli
space non-trivial $SL(2, \C)$ flat connections on Seifert integer homology
spheres, and in particular, Brieskorn homology spheres from
\cite[Proposition~5]{andersen2022resurgence} (see also
\cite[\S~4]{Hikami05Brieskorn}). Let $\Sigma(p_1, p_2, p_3)$ be a
Brieskorn homology sphere, and suppose (possibly after relabeling) that $p_2$
and $p_3$ are odd. Define $A_{\vec p} \subset \Z_{>0}^3$ as the set of
triples $(a_1, a_2, a_3)$ such that $1 \leq a_1 < p_1$ and $1 \leq a_j
\leq (p_j - 1)/2$ for $j = 2$ and $3$, then there is an isomorphism
\begin{equation}
    \pi_0 (\mathcal M\ped{flat}\ap{non-ab}(\Sigma(p_1, p_2, p_3), SL(2, \C)))
    \cong A_{\vec p}.
\end{equation}
Each element $\vec a \in A_{\vec p}$ is called the \emph{rotation number}
of the flat connection it represents, and the Chern--Simons value of the
corresponding flat connection is
\begin{equation}
    \CS[\vec a] = - \frac{P}{4}\left(1 + \sum_{j = 1}^3
    \frac{a_j}{p_j}\right)^2 \bmod 1.
\end{equation}

Note that the set $A_{\vec p}$ has cardinality 
\begin{equation}
    D_{\vec p} = \frac{1}{4}(p_1 - 1) (p_2 - 1) (p_3 - 1),
\end{equation}
and the triples that are its elements, $\vec a \in A_{\vec p}$, pick out the
$D_{\vec p}$ independent false theta functions $\tilde\Phi^{\vec a}_{\vec p}$.
With this correspondence, we can write the $T$-matrix of the modular
transformation as $T^{\vec a}(\vec p) = e^{-2\pi i \CS[\vec a]}$.
And the following result says that the limit of the false theta function
$\tilde\Phi^{\vec a}_{\vec p}(s/r)$ at a rational number is a polynomial in the
root of unity $\xi = e^{\frac{2\pi i s}{r}}$ with an overall fractional power that
equals $\CS[\vec a]$.
 
\newcommand{\intcoeffprop}{
    Let $\xi = e^{\frac{2\pi i s}{r}}$ be a root of unity of odd order, then
    \begin{equation}
        \frac{1}{2} \tilde \Phi^{\vec a}_{\vec p}(s/r) \in \xi^{-\CS[\vec a]}
        \Z[\xi].
    \end{equation}
}

\begin{proposition}
    \label{prop:int-coeff}
    \intcoeffprop
\end{proposition}
\begin{proof}
    See Appendix~\ref{app:int-coeff}.
\end{proof}

\paragraph{WRT invariant as limits of false theta functions and proof of conjecture}

Finally, we have all the tools to write the WRT invariant as limits of false
theta functions and prove the quantum modularity conjecture for Brieskorn
homology spheres.

\begin{proposition}
    \label{prop:brieskorn-false-theta}
    For Brieskorn homology spheres, $\Sigma(p_1, p_2, p_3)$ with $1/p_1 + 1/p_2
    + 1/p_3 < 1$, the WRT invariant at a generic primitive root of unity of odd
    order, $\xi = e^{\frac{2\pi i s}{r}}$, can be written as a limit of the false theta
    function,
    \begin{equation}
        \xi^{\phi/4 - 1/2}(\xi - 1) \tau_{\Sigma(p_1, p_2, p_3)} (\xi) =
        \frac{1}{2}\tilde\Phi_{\vec p}^{(1, 1, 1)} (s/r).
    \end{equation}
\end{proposition}

\begin{proof}
    The proof here follows from results in \cite{LZ99}, and is an adaptation of
    the one given in \cite{Hikami05Brieskorn} to the case of generic roots of
    unity. We recall the formula (\ref{eqn:wrt-seifert}) for the WRT invariant
    \begin{equation}
        \label{eqn:wrt-brieskorn}
        \xi^{\phi/4 - 1/2} (\xi - 1) \tau_{\Sigma(p_1, p_2, p_3)}(\xi) =
        \left(\frac{Pr}{s}\right) \frac{e^{\pi i / 4}}{2\sqrt{2Pr}}
        \quad\sum_{\mathclap{\substack{n \in \Z/2Pr\Z \\ r \nmid n}}}\ \xi^{-n^2/4P}
        \frac{\prod_{j = 1}^3 (\xi^{n/2p_j} - \xi^{-n/2p_j})}{\xi^{n/2} -
        \xi^{-n/2}},
    \end{equation}
    and use the fact that for $1/p_1 + 1/p_2 + 1/p_3 < 1$, there is a
    generating function for the rational function in the summand above in
    terms of the periodic function $\phi^{(1, 1, 1)}_{\vec p}$ from before,
    \begin{equation}
        \frac{\prod_{j = 1}^3 (u^{P/p_j} - u^{-P/p_j})}{u^{P} -
        u^{-P}} = \sum_{l = 0}^{\infty} \phi^{(1, 1, 1)}_{\vec p}(l)\, u^l,
    \end{equation}
    to write it as a limit with $u = \xi^{n/2P} e^{-\pi t/2P}$,
    \begin{equation}
        \frac{\prod_{j = 1}^3 (\xi^{n/2p_j} - \xi^{-n/2p_j})}{\xi^{n/2} -
        \xi^{n/2}} = \lim_{t \to 0^+} \sum_{l = 0}^{\infty} \phi^{(1, 1,
        1)}_{\vec p}(l)\, u^l.
    \end{equation}
    For the sum in (\ref{eqn:wrt-brieskorn}), we have
    \begin{align}
        \sum_{\mathclap{\substack{n \in \Z/2Pr\Z \\ r \nmid n}}} \ 
        \xi^{-n^2/4P} \frac{\prod_{j = 1}^3 (\xi^{n/2p_j} -
        \xi^{-n/2p_j})}{\xi^{n/2} - \xi^{-n/2}} 
        & = \lim_{t \to 0^+} \sum_{l = 0}^\infty e^{-\pi t l/2P} 
        \phi^{(1, 1, 1)}_{\vec p}(l)\ \sum_{\mathclap{n \in \Z/2Pr\Z}}\
        \xi^{-n^2/4P + nl/2P} \nonumber \\
        & = \lim_{t \to 0^+} \sum_{l = 0}^\infty e^{-\pi t l/2P} \phi^{(1, 1,
        1)}_{\vec p}(l)\ \sum_{\mathclap{n \in \Z/2Pr\Z}}\ \xi^{-(n - l)^2/4P +
        l^2/4P} \nonumber \\
        & = \ \sum_{\mathclap{n \in \Z/2Pr\Z}}\ \xi^{-n^2/4P} \lim_{t \to 0^+} 
        \sum_{l = 0}^\infty e^{-\pi t l/2P} \phi^{(1, 1, 1)}_{\vec p}(l)
        \,\xi^{l^2/4P},
        \nonumber
    \end{align}
    where we have used a square completion to simplify the Gauss sum.
    To proceed, we can evaluate the quadratic Gauss sum in the last line above
    with a formula from Appendix~\ref{app:gauss-sums},
    \begin{equation}
        \sum_{n \in \Z/2Pr} \xi^{-n^2/4P} = \frac{1}{2} \sum_{n \in \Z/4Pr}
        \xi^{-n^2/4P} = \left(\frac{Pr}{s}\right) e^{-\pi i /4} \sqrt{2Pr},
    \end{equation}
    (remember that we have assumed
    $s = 1 \bmod 4$), and conclude by using a proposition from \cite{LZ99} for
    the first equality in the following
    \begin{align}
        \lim_{t \to 0^+} \sum_{l = 0}^\infty e^{-\pi t l/2P} \phi^{(1, 1, 1)}_{\vec p}(l)
        \,\xi^{l^2/4P} & = \lim_{t \to 0^+} \sum_{l = 0}^\infty e^{-\pi t l^2/2P} 
        \phi^{(1, 1, 1)}_{\vec p}(l)\, \xi^{l^2/4P} \\
        & = \lim_{t \to 0^+} \tilde \Phi^{(1, 1, 1)}_{\vec p}(s/r + it) =
        \tilde \Phi^{(1, 1, 1)}_{\vec p}(s/r).
    \end{align}
    Putting all the pieces together, the result follows.
\end{proof}

The only Brieskorn homology sphere for which $1/p_1 + 1/p_2 + 1/p_3 < 1$ does
not hold is the Poincar\'e homology sphere, $\Sigma(2, 3, 5)$, and we will deal
with it separately here. When $\vec p = (2, 3, 5)$, the only difference from
the above calculation is due to
\begin{equation}
    \frac{\prod_{j = 1}^3 (u^{P/p_j} - u^{-P/p_j})}{u^P - u^{-P}} = u + u^{-1}
    + \sum_{l = 0}^\infty \phi^{(1, 1, 1)}_{(2, 3, 5)} (l)\, u^l,
\end{equation}
which results in an extra constant in the analogous formula for the WRT
invariant,
\begin{equation}
    \xi^{121/120}(\xi - 1) \tau_{\Sigma(2, 3, 5)}(\xi) = \xi^{1/120} +
    \frac{1}{2} \tilde\Phi^{(1, 1, 1)}_{(2, 3, 5)}(s/r).
\end{equation}

\begin{theorem}
\label{thm:bri}
    The quantum modularity conjecture holds for Brieskorn homology spheres.
\end{theorem}

\begin{proof}
    With the exception of $\Sigma(2, 3, 5)$ which is spherical, all other
    Brieskorn homology spheres $M = \Sigma(p_1, p_2, p_3)$ have $1/p_1 + 1/p_2
    + 1/p_3 < 1$ and are modeled on $\widetilde{SL(2, \R)}$. We start by
    writing the WRT invariant of these manifolds in terms of the false theta
    function as in Proposition~\ref{prop:brieskorn-false-theta},
    \begin{equation}
        W_{M}(\xi) = (\xi - 1)\, \tau_{M}(\xi) = \frac{1}{2} \xi^{1/2 - \phi/4}\,
        \tilde\Phi^{(1, 1, 1)}_{\vec p}(s/r),
    \end{equation}
    use the modular $S$-transformation of the false theta functions, and
    arrange the terms into an expansion around saddles in the following way
    \begin{align}
        W_{M}(\xi) & \simeq -\frac{1}{2} \sqrt{\frac{s}{ir}}
        \xi^{1/2 - \phi/4} \sum_{\vec a} S^{(1, 1, 1)}_{\vec a}(\vec p) \tilde
        \Phi^{\vec a}_{\vec p} (-r/s) + \frac{1}{2} \xi^{1/2 - \phi/4} \sum_{k
        = 0}^\infty \frac{a_k}{k!} \left(\frac{\pi i s}{2Pr}\right)^k \\
        & = \sum_{A \in \pi_0 (\mathcal M\ped{flat}(M, SL(2, \C)))} e^{2\pi
        i \frac{r}{s} \CS[A]} P_{A}(\tilde \xi)\, I_{A}(s/r), 
        \label{eqn:brieskorn-saddles}
    \end{align}
    where $P\ped{triv}(\tilde \xi) = 1$, and
    \begin{equation}
        P_A(\tilde \xi) = \frac{1}{2} \tilde \xi^{\CS[\vec a]} \tilde
        \Phi^{\vec a}_{\vec p}(-r/s)
    \end{equation}
    if $A$ is the flat connection represented by rotation numbers $(a_1, a_2,
    a_3)$; and
    \begin{equation}
        I\ped{triv}(s/r) = \frac{1}{2} \xi^{\frac{1}{2} - \frac{\phi}{4}}
        \sum_{k = 0}^\infty \frac{L(-2k, \phi^{(1, 1, 1)}_{\vec p})}{k!}
        \left(\frac{\pi i s}{2Pr}\right)^k,
    \end{equation}
    and
    \begin{equation}
        I_A(s/r) = -\sqrt{\frac{r}{is}}\, S^{(1, 1, 1)}_{\vec a}(\vec p) \xi^{1/2 -
        \phi/4}
    \end{equation}
    if $A$ is the flat connection represented by rotation numbers $(a_1, a_2,
    a_3)$.

    We have identified the asymptotic series in $s/r$ that appears as the
    failure of modularity in the $S$-transformation of $\tilde\Phi^{(1, 1,
    1)}_{\vec p}$ as the contribution from the trivial flat connection, and
    used the correspondence between non-trivial flat connections and the
    $D_{\vec p}$ independent theta functions to write the sum over the false
    theta functions as a sum over the connected components of $\mathcal
    M\ped{flat}(M, SL(2, \C))$ in (\ref{eqn:brieskorn-saddles}).
 
    Asymptotic series $I_A$ are of the expected type, and the integrality of
    the polynomials $P_A(\tilde \xi) \in \Z[\tilde \xi]$ is implied directly by
    Proposition~\ref{prop:int-coeff}.

    Finally, it follows from (\ref{CS-geom-sl2r}) that, for the geometric flat
    connection on $\widetilde{SL(2, \R)}$ manifolds, $\CS[A_*] = -\chi^2/4e$.
    Comparing the Chern--Simons values tells us that $A_*$ is represented by
    rotation numbers $(1, 1, 1)$, and therefore it follows from
    Propositions~\ref{prop:phi-cs} and \ref{prop:brieskorn-false-theta} that
    $P_{A_*}(\tilde \xi) = \xi^\delta W_M(\tilde \xi)$ for some integer
    $\delta$.

    The case of $\Sigma(2, 3, 5)$ is completely analogous, and we have
    \begin{align}
        W_{\Sigma(2, 3, 5)}(\xi) & = \xi^{-1} + \frac{1}{2} \xi^{-121/120} \tilde
        \Phi^{(1, 1, 1)}_{(2, 3, 5)}(s/r) \\
        & \simeq \sum_{A \in \pi_0 (\mathcal M\ped{flat}(M, SL(2, \C)))} e^{
            \frac{2 \pi i r}{s} \CS[A]} P_A(\tilde \xi) I_A(s/r),
    \end{align}
    with $P\ped{triv}(\tilde\xi) = 1$, and
    \begin{equation}
        P_A(\tilde \xi) = \frac{1}{2} \tilde \xi^{-1/120} \tilde\Phi^{\vec
        a}_{\vec p}(-r/s),
    \end{equation}
    if $A$ is the flat connection represented by rotation numbers $(a_1, a_2,
    a_3)$; and
    \begin{equation}
        I\ped{triv}(s/r) = \xi^{-1} + \frac{1}{2} \xi^{-121/120} \sum_{k =
        0}^\infty \frac{L(-2k, \phi^{(1, 1, 1)}_{(2, 3, 5)})}{k!}
        \left(\frac{\pi i s}{60 r}\right)^k
    \end{equation}
    and
    \begin{equation}
        I_A(s/r) = - \sqrt{\frac{r}{is}} S^{(1, 1, 1)}_{\vec a} \xi^{-121/120}
    \end{equation}
    if $A$ is the flat connection represented by rotation numbers $(a_1, a_2,
    a_3)$. We identify the geometric flat connection by its Chern--Simons
    value, $\CS[A_*] = -1/120$, and notice that
    \begin{equation}
        P_{A_*}(\tilde \xi) = \frac{1}{2} \tilde\xi^{-1/120}
        \tilde\Phi^{(1,1,1)}_{(2, 3, 5)}(-r/s)
        = \tilde \xi W_{\Sigma(2, 3, 5)}(\tilde \xi) - 1,
    \end{equation}
    as required.
\end{proof}

\subsection{More evidence}

\paragraph{Lens spaces $L(p, 1)$} 

For any positive integer $p$, the lens space $L(p, 1)$ has a realization as surgery on
the $p$-framed unknot in $S^3$.  Computation of the WRT invariant of these
spaces is essentially an application of the reciprocity of quadratic Gauss sums. We
start by noting that $J_n(U^p, q) = q^{p(n^2 - 1)/4}[n]_q$, and
therefore for $L(p, 1) = S^3(U^p)$, we need the intermediate quantity
\begin{align}
    F(U^p, \xi) & = \sum_{n = 1}^{r - 1} \xi^{p\frac{n^2 - 1}{4}} [n]^2_\xi =
    \frac{1}{2} \sum_{n = -r}^{r - 1} \xi^{p\frac{n^2 - 1}{4}} [n]_\xi^2 =
    \frac{1}{2} \sum_{n \in \Z/2r\Z} \xi^{p\frac{n^2 - 1}{4}} [n]^2_\xi \\
    & = \frac{\xi^{-p/4}}{2(\xi^{1/2} - \xi^{-1/2})^2} \sum_{n \in \Z/2r\Z}
    \xi^{\frac{pn^2}{4}} (\xi^n + \xi^{-n} - 2),
\end{align}
which can be simplified by applying the reciprocity formula as
\begin{equation}
    \sum_{n \in \Z/2r\Z} \xi^{\frac{pn^2}{4}} = \frac{e^{\pi i /4}
    \sqrt{2r}}{\sqrt{sp}} \sum_{y \in \Z/sp\Z} e^{-\frac{2\pi i r}{sp}y^2} =
    \frac{e^{\pi i /4} \sqrt{2r}}{\sqrt{sp}} \sum_{l \in \Z/s\Z}
    \tilde\xi^{pl^2} \sum_{a \in \Z/p\Z} e^{2\pi i \frac{r}{s}
    \left(-\frac{(sa)^2}{p}\right)},
\end{equation}
and
\begin{align}
    \sum_{n \in \Z/2r\Z} \xi^{\frac{pn^2}{4} \pm n} & = \frac{e^{\pi i /4}
    \sqrt{2r}}{\sqrt{sp}} \sum_{y \in \Z/sp\Z} e^{-\frac{2\pi i r}{sp} \left(y
    \pm \frac{s}{r}\right)^2} \\
    & = \frac{e^{\pi i /4} \sqrt{2r}}{\sqrt{sp}}
    \sum_{l \in \Z/s\Z} \tilde \xi^{pl^2} \sum_{a \in \Z/p\Z} e^{2\pi
    i \frac{r}{s} 
    \left(-\frac{(sa)^2}{p}\right) \mp \frac{4\pi i sa }{p}} \xi^{-1/p},
\end{align}
where we have used $(p, s) = 1$ to decompose the sum over $\Z/sp\Z \cong \Z/s\Z
\oplus \Z/p\Z$. After this decomposition, the sum above over $l$ is precisely a
quadratic Gauss sum, for which there is a closed expression
\begin{equation}
    \sum_{l \in \Z/s\Z} \tilde\xi^{pl^2} = \overline{G(rp, s)} = \left(\frac{rp}{s}\right)
    \sqrt{s},
\end{equation}
where we have used the fact that $s = 1 \bmod 4$. We can put these pieces
together, and normalize with $F(U^{+1}, \xi)$ for the WRT invariant
\begin{equation}
    \xi^{\frac{1 - p}{4}} W_{L(p, 1)}(\xi)  =
    \sum_{\mathclap{a \in \Z/p\Z}}\ e^{2\pi i \frac{r}{s}
    \left(-\frac{s^2 a^2}{p}\right)} \left(\xi^{-1/p} \cos\frac{4\pi s a}{p} - 1
    \right).
\end{equation}
Since, under our assumptions, $s$ is coprime with $\lvert H_1 \rvert=p$, there exists an integer $s^*$ such that $ss^*=1+n\,p$ for some ($s$-dependent) $n\in\Z$. After making the change $a\mapsto s^*a$ of the summation variable, we can read off the decomposition into abelian flat connections as
\begin{align}
    W_{L(p, 1)}(\xi) & = \xi^{\frac{p - 1}{4}} \sum_{a = 0}^{p - 1} e^{2\pi i
    \frac{r}{s} \left(-\frac{a^2}{p}+2na^2+n^2pa\right)} \left(\xi^{-1/p} \cos \frac{4\pi
    a}{p} - 1 \right) \\
    & = \sum_{a \in \pi_0 (\mathcal
    M\ped{flat}\ap{ab}(L(p, 1), SL(2, \C)))} e^{2\pi i \frac{r}{s} \CS[a]}
    W_{L(p, 1)}^{a}(\xi),
\end{align}
where, as before, $\CS[a]$ means a lift of the CS value to $\Q$. Using $\pi_0 (\mathcal M\ped{flat}\ap{ab}(L(p, 1), SL(2, \C))) \cong \{0,
1, \ldots, (p - 1)/2 \}$, we have
\begin{equation}
    W^{a}_{L(p, 1)}(\xi) = 
    \begin{cases}
        \displaystyle \xi^{\frac{p - 1}{4}} \left(\xi^{-\frac{1}{p}} - 1\right) & \text{if $a
        = 0$}, \\
        \displaystyle 2 \xi^{\frac{p - 1}{4}} \left(\xi^{-\frac{1}{p}} \cos \frac{4\pi s a}{p} - 1
        \right) & \text{if $a \neq 0$}.
    \end{cases}
\end{equation}

Since the geometric flat connection for lens spaces is abelian and has CS value
$-1/p$, we have $P\ap{triv}_{A_*}(\tilde \xi) = 0$ because of the way the
decomposition is defined. On the other hand, since
\begin{equation}
    \sum_{a = 0}^{(p - 1)/2} W^{a}_{L(p, 1)}(\tilde \xi) = -p \xi^{\frac{p -
    1}{4}},
\end{equation}
we verify that
\begin{equation}
    P\ap{triv}_{A_*}(\tilde \xi) = \xi^{\frac{1 - p}{4}} \sum_{a \in \pi_0 (\mathcal
    M\ped{flat}\ap{ab}(L(p, 1), SL(2, \C)))} W^{a}_{L(p, 1)}(\tilde \xi) + p.
\end{equation}
Thus the conjecture holds for these lens spaces.

For some non-trivial examples, we can carry out the analysis we did for
Brieskorn homology spheres to some Seifert rational homology spheres. To do so, it would be useful to introduce another basis of false theta functions.
Let $P$ be a positive integer, and for every integer $1 \leq a < P$ define
$2P$-periodic functions
\begin{equation}
    \psi^{(a)}_{2P}(l) = 
    \begin{cases}
        \pm 1 & \text{if $l = \pm a \bmod 2P$}, \\
        0 & \text{otherwise}.
    \end{cases}
\end{equation}
Each $\psi^{(a)}_{2P}$ has an associated theta function
\begin{equation}
    \Psi^{(a)}_P(\tau) = \sum_{l \in \Z} l \psi^{(a)}_{2P}(l) q^{l^2/4P}
    \quad\text{with $q = e^{2\pi i \tau}$}
\end{equation}
defined for $\Im \tau > 0$, which are components of a vector-valued modular
form of weight $3/2$. Under $T$ and $S$ modular transformations, they behave as
\cite[\S~2.2]{Hikami05}
\begin{equation}
    \Psi^{(a)}_P(\tau + 1) = e^{\pi i a^2/2P} \Psi^{(a)}_P(\tau)
    \quad\text{and}\quad
    \Psi^{(a)}_P(\tau) = \left(\frac{i}{\tau}\right)^{3/2} \sum_{b = 1}^{P - 1}
    M^a_b(P) \Psi^{(b)}_P(-1/\tau)
\end{equation}
respectively, where
\begin{equation}
    M^a_b(P) = \sqrt{\frac{2}{P}} \sin \frac{ab}{P}\pi.
\end{equation}
In analogy with what we did for the modular forms $\Phi^{\vec a}_{\vec p}$, we
define false theta functions, which are the Eichler integrals
\begin{equation}
    \tilde\Psi^{(a)}_P(\tau) = \sum_{l = 0}^\infty \psi^{(a)}_{2P}(l)
    q^{l^2/4P},
\end{equation}
whose limit at rational numbers
\begin{equation}
    \tilde\Psi^{(a)}_P(s/r) = \frac{1}{2} \sum_{l = 0}^{2Pr} \psi^{(a)}_{2P}(l)
    \xi^{l^2/4P} \left(1 - \frac{l}{Pr}\right),
\end{equation}
satisfies the following property under a modular $S$-transformation
\cite[Proposition~2]{Hikami05}
\begin{equation}
    \tilde\Psi^{(a)}_P(s/r) + \sqrt{\frac{r}{is}} \sum_{b = 1}^{P - 1} M^a_b(P)
    \tilde\Psi^{(b)}_P(-r/s) \simeq \sum_{k = 0}^\infty \frac{L(-2k,
    \psi^{(a)}_{2P})}{k!} \left(\frac{\pi i s}{2Pr}\right)^k.
\end{equation}

We shall consider rational homology spheres with three exceptional fibers $M =
S^2(b;\, p_1/q_1, p_2/q_2, p_3/q_3)$, where $q_j = \pm 1$ for every $j = 1, 2,
3$. In what follows, we use $p_j$ to denote $p_j/q_j$.  Such manifolds are
described as surgery on a link of unknots that has the linking matrix
\begin{equation}
    B = \begin{bmatrix}
        b & 1 & 1 & 1 \\
        1 & p_1 & 0 & 0 \\
        1 & 0 & p_2 & 0 \\
        1 & 0 & 0 & p_3
    \end{bmatrix},
\end{equation}
and the $n = (n_0, n_1, n_2, n_3)$-colored Jones polynomial
\begin{align}
    J_{n}(L, q) 
    & = \frac{q^{b(n_0^2 - 1)/4}}{[n_0]_q^2} \prod_j q^{p_j(n_j^2 - 1)/4} [n_0
    n_j]_q \\
    & = \frac{1}{q^{1/2} - q^{-1/2}} \frac{q^{b(n_0^2 - 1)/4}}{(q^{n_0/2} -
    q^{-n_0/2})^2} \prod_j q^{p_j(n_j^2 - 1)/4} (q^{n_0 n_j / 2} - q^{n_0 n_j /
    2}).
\end{align}
Using the surgery formula for WRT invariants, we get
\begin{align}
    \xi^{\phi'/4 - 1/2} (\xi - 1) \tau_M(\xi) = (-1)^{b_+} \frac{e^{-\pi
    i \sigma(B)/4}}{2(2r)^2} & \sum_{\substack{n_0 \in \Z/2r\Z \\ n_0 \neq 0, r}}
    \frac{\xi^{bn_0^2/4}}{\xi^{n_0/2} - \xi^{-n_0/2}} \nonumber \\
    & \times \prod_j \sum_{n_j \in
    \Z/2r\Z} \xi^{p_j n_j^2/4 + n_0 n_j/2} \left(\xi^{n_j/2} -
    \xi^{-n_j/2}\right),
    \nonumber
\end{align}
where $\phi' = \Tr B - 3\sigma(B)$.
To get the above formula in a form similar to (\ref{eqn:wrt-seifert}), we
apply the reciprocity formula for quadratic Gauss sums
(Theorem~\ref{thm:reciprocity}) to the sums over $n_j$, and use the assumption
$(s, p_j) = 1$ on $s$ to simplify, so that
\begin{equation}
    \xi^{\phi/4 - 1/2} (\xi - 1) \tau_M(\xi) = - \frac{(-1)^{b_+}}{2
    \sqrt{2\lvert P \rvert r}} \left(\frac{Pr}{s}\right) \frac{e^{\frac{\pi
    i}{4}\sum_j \sigma(p_j)}}{e^\frac{\pi i \sigma(B)}{4}} \Bigg( \cdots
    \Bigg),
\end{equation}
where $\phi = \phi' + \sum_j\frac{1}{p_j}$ and
\begin{equation}
    \Bigg(\cdots\Bigg) = 
    \sum_{\substack{n_0 \in \Z/2r\Z \\ n_0 \neq 0, r}}
    \frac{\xi^{bn_0^2/4}}{\xi^{n_0/2} - \xi^{n_0/2}} \prod_j \sum_{n_j \in
    \Z/p_j\Z} \xi^{-\frac{(n_0 + 2rn_j)^2}{4p_j}} \left(\xi^{\frac{n_0 +
    2rn_j}{2p_j}} - \xi^{-\frac{n_0 + 2rn_j}{2p_j}}\right).
\end{equation}
Following \cite[Proposition~2]{Hikami06Spherical}, we notice that the sum above
is invariant under simultaneous change $n_0 \to n_0 + 2r$ and $n_j \to n_j - 1$
for every $j = 1, 2, 3$. If $p_1$ and $p_2$ are coprime, we can use this
invariance to simplify the sums
\begin{equation}
    \sum_{n_0 = 0}^{2r - 1} \sum_{n_1 \in \Z/p_1\Z,\  n_2 \in \Z/p_2\Z,\  n_3 \in \Z/p_3\Z} =
    \sum_{n_0 = 0}^{2 p_1 p_2 r - 1} \sum_{n_1 = 0,\  n_2 = 0,\  n_3 \in \Z/p_3\Z}.
\end{equation}

In the following examples, we shall use the above formula for the WRT invariant
at generic roots of unity to write it in terms of false theta functions.

\paragraph{$S^2(1;\, 2, 3, 3)$} For the Seifert manifold $S^2(1;\, 2, 3, 3)$,
we can readily compute
\begin{equation}
    e = \frac{1}{6}, \quad
    \chi = \frac{1}{6}, \quad
    \phi = \frac{25}{6}, \quad
    H_1(M, \Z) \cong \coker B \cong \Z/3\Z.
\end{equation}
Since $e \neq 0$ and $\chi > 0$, this manifold must be spherical.
Indeed, it can also be thought of as the quotient of $S^3$ by the binary
tetrahedral group. Apart from the Poincar\'e homology sphere and lens spaces
$L(p, 1)$ for odd $p$, this is the only ADE type spherical manifold that is a
mod-2 homology sphere.

Its WRT invariant is given by $\xi^{13/24}(\xi - 1) \tau_M(\xi) =$
\begin{align}
    \frac{1}{6\sqrt r} \left(\frac{18r}{s}\right) \sum_{\substack{n_0 \in
    \Z/12r\Z \\ r \nmid n_0}} \xi^{-n_0^2/24} & \frac{(\xi^{n_0/4} -
    \xi^{-n_0/4}) (\xi^{n_0/6} - \xi^{-n_0/6})}{\xi^{n_0/2} - \xi^{-n_0/2}}
    \nonumber \\
    & \times \sum_{n_3 \in \Z/3\Z} \xi^{\frac{-r^2 n_3^2 + rn_3n_0}{3}} \left(\xi^{\frac{n_0 +
    2rn_3}{6}} - \xi^{-\frac{n_0 + 2rn_3}{6}} \right).
\end{align}
Analogously to the case of Brieskorn homology spheres, we can proceed by
writing the rational function in the summand above as the limit of a holomorphic
function on the upper half plane, and then complete squares in the exponent of
$\xi$ to simplify the Gauss sum over $n_0$ (see \cite{Hikami06Spherical} for a
similar calculation for the special root of unity $\xi = e^{\frac{2\pi i}{r}}$). After
a bit of work, and, as in the case of lens spaces, the fact that $s$ is coprime with 3,  we can show that
\begin{align}
    \xi^{13/24} W_M(\xi) 
    & = -\frac{1}{2} \sum_{a \in \Z/3\Z} e^{2\pi i
    \frac{r}{s} \left(\frac{s^2 a^2}{3}\right)} \left(\tilde\Psi^{(1) + (5)}_6(s/r) +
    2\xi^{ra/3} \tilde\Psi^{(3)}_6(s/r) - 2 \xi^{1/24}\right) \\
    & =\sum_{a \in \{0, 1\} \cong \pi_0 (\mathcal M\ped{flat}\ap{ab}(M,
    SL(2, \C)))}
    e^{2\pi i \frac{r}{s} \CS[a]} W_M^{a}(\xi)
\end{align}
where we have used the shorthand notation $\tilde\Psi^{n_a (a) + n_b (b) +
\cdots} = n_a \Psi^{(a)} + n_b \Psi^{(b)} + \cdots$. Components of the
decomposition are
\begin{align}
    \xi^{13/24} W_M^{0}(\xi) & = - \frac{1}{2} \tilde\Psi^{(1) + 2(3) + (5)}_6 (s/r) +
    \xi^{1/24} \quad\text{and}\quad \\
    \xi^{13/24} W_M^{1}(\xi) & = - \tilde\Psi^{(1) - (3) + (5)}_6 (s/r) +
    2\xi^{1/24}.
\end{align}
Asymptotic behavior of the above functions as $r \to \infty$ follows from the
modular $S$-transformation of false theta functions, for example\footnote{We can label the components of the moduli space of flat connections by the respective CS values, when they are uniquely determined by them, as in this case.},
\begin{align}
    W^{0}_M(\xi) & \simeq -\frac{1}{2} \tilde\Psi^{3(1) + 3(5)}_6(-r/s)
    \left(-\sqrt{\frac{r}{3is}} \xi^{-13/24} \right) + \xi^{-1/2} -
    \xi^{-13/24} \sum_{k = 0}^\infty \frac{c_k}{k!}\left(\frac{\pi is
    }{12r}\right)^k \\
    & = e^{2\pi i \frac{r}{s} \left(-\frac{1}{24}\right)}
    P^{0}_{-1/24}(\tilde \xi) I^{0}_{-1/24}(s/r) + P^{0}_0(\tilde \xi)
    I^{0}_0(s/r),
\end{align}
where we have denoted the flat connection by its CS value and arranged the
terms in the following way
\begin{equation}
    P^{0}_{-1/24}(\tilde \xi) = -\frac{1}{2} \tilde\xi^{-1/24}
    \tilde\Psi^{3(1) + 3(5)}_6(-r/s)
    \quad\text{and}\quad
    I^{0}_{-1/24}(s/r) = - \sqrt{\frac{r}{3is}} \xi^{-13/24},
\end{equation}
and for the trivial flat connection
\begin{equation}
    P_0^{0}(\tilde \xi) = 1
    \quad\text{and}\quad
    I_0^{0}(s/r) = \xi^{-1/2} - \xi^{-13/24} \sum_{k = 0}^\infty
    \frac{c_k}{k!}\left(\frac{\pi i s}{12 r}\right)^k,
\end{equation}
with coefficients $c_k = L(-2k, \psi^{(1) + 2(3) + (5)}_{12})$. Similarly, it
can be verified that the asymptotic expansion for $W^{1}_M(\xi)$ is of the
required form.

Finally, we recall that for the geometric flat connection $A_*$, we have
$\CS[A_*] = -1/24$, and a direct calculation shows that
\begin{equation}
    P^{0}_{-1/24}(\tilde \xi) = \tilde\xi^{1/2} \sum_{a \in \pi_0 (\mathcal
    M\ped{flat}\ap{ab}(M, SL(2, \C)))} W^{a}_M(\tilde \xi) - 3.
\end{equation}

Thus, we have verified that the Conjecture~\ref{conj:qhs} holds for this
manifold.

\paragraph{$S^2(-1;\, -2, -3, -9)$}
For the Seifert manifold $S^2(-1;\, -2, -3, -9)$, we calculate
\begin{equation}
    e = \frac{1}{18}, \quad
    \chi = -\frac{1}{18}, \quad
    \phi = -\frac{71}{18}, \quad
    H_1(M, \Z) = \Z/3\Z.
\end{equation}
Since $e \neq 0$ and $\chi < 0$, this manifold must be modeled on
$\widetilde{SL(2, \R)}$ geometry. To verify that the conjecture holds for this
manifold, we proceed as in the previous example and write the WRT invariant as
in terms of false theta functions. We start with the decomposition
\begin{align}
    \xi^{-107/72} W_M(\xi) & = \frac{1}{2} \sum_{a \in \Z/3\Z} e^{2\pi i \frac{r}{s}
    \left(\frac{s^2 a^2}{3}\right)} \left(\tilde\Psi^{(1) + (17)}_{18}(s/r) -
    \xi^{ra/3} \tilde\Psi^{(5) + (13)}_{18}(s/r) \right) \\
    & = \sum_{a \in \{0, 1\} \cong \pi_0 (\mathcal M\ped{flat}\ap{ab}(M, SL(2, \C)))}
    e^{2\pi i\frac{r}{s} \CS[a]} \xi^{-107/72} W^{a}_M(\xi),
\end{align}
where
\begin{align}
    \xi^{-107/72} W^{0}_M(\xi) & = \frac{1}{2}\tilde\Psi^{(1) - (5) - (13) +
    (17)}_{18}(s/r) \quad\text{and} \\
    \xi^{-107/72} W^{1}_M(\xi) & = \frac{1}{2}\tilde\Psi^{2(1) + (5) + (13) +
    2(17)}_{18}(s/r).
\end{align}
As before, we can compute the asymptotic expansion of the functions $W^{a}_M$
by using the modular transformation property of false theta functions
\begin{equation}
    W^{0}_M(\xi) \simeq \sum_{A \in \pi_0 (\mathcal M\ped{flat}(M, SL(2, \C)))}
    e^{2\pi i \frac{r}{s} \CS[A]} P^{0}_A(\tilde \xi) I^{0}_A(s/r),
\end{equation}
where
\begin{align*}
    P^{0}_{-1/72}(\tilde \xi) & = \frac{1}{2} \tilde\xi^{-1/72} \tilde\Psi^{3(1)
    + 3(17)}_{18}(-r/s),\quad
    I^{0}_{-1/72} = -\frac{2}{9} \sqrt{\frac{r}{is}} \left(\sin \frac{\pi}{18} -
    \sin\frac{5\pi}{18}\right) \xi^{107/72}, \\
    P^{0}_{-25/72}(\tilde \xi) & = \frac{1}{2} \tilde\xi^{-25/72} \tilde\Psi^{3(5)
    + 3(13)}_{18}(-r/s),\quad
    I^{0}_{-25/72} = -\frac{2}{9} \sqrt{\frac{r}{is}} \left(\sin \frac{5\pi}{18} +
    \sin\frac{7\pi}{18}\right) \xi^{107/72}, \\
    P^{0}_{-49/72}(\tilde \xi) & = \frac{1}{2} \tilde\xi^{-49/72} \tilde\Psi^{3(7)
    + 3(11)}_{18}(-r/s),\quad
    I^{0}_{-49/72} = -\frac{2}{9} \sqrt{\frac{r}{is}} \left(\sin \frac{\pi}{18} +
    \sin\frac{7\pi}{18}\right) \xi^{107/72},
\end{align*}
for the trivial flat connection, we have
\begin{equation}
    P^{0}\ped{triv}(\tilde \xi) = 1,\quad
    I^{0}\ped{triv}(s/r) = \sum_{k = 0}^\infty \frac{L(-2k, \psi^{(1) - (5) - (13) +
    (17)}_{18})}{k!} \left(\frac{\pi i s}{36r}\right)^k,
\end{equation}
as usual, and $P^{0}_{-19/24}(\tilde \xi) = 0$ because $W^{0}_M(\xi)$ does
not receive any contribution from the non-abelian flat connection with $\CS =
-19/24$. It can be verified that the other piece $W^{1}_M(\xi)$ also has an
asymptotic expansion of the expected type.

Finally, since for the geometric flat connection $\CS[A_*] = -1/72 \bmod 1$, we
can verify that 
\begin{equation}
    P_{-1/72}(\tilde \xi) = \tilde\xi^{-3/2} \sum_{a \in \pi_0 (\mathcal
    M\ped{flat}\ap{ab}(M, SL(2, \C)))} W^{a}_M(\xi).
\end{equation}

\paragraph{Seifert manifolds $S^2(0;\, p, -2p - 1, -2p - 1)$} As a final
example, we consider the family of Seifert manifolds $M = S^2(0;\, p, -2p - 1, -2p
- 1)$ for positive integers $p \geq 2$. We can calculate
\begin{equation}
    e = \frac{1}{p(2p + 1)},\quad \chi = -\frac{2p^2 - 3p - 1}{p (2p +
    1)}, \quad H_1(M, \Z) = \Z/(2p + 1)\Z.
\end{equation}
These manifolds are mod-2 homology spheres, and since $e \neq 0$ and $\chi < 0$
they must be modeled on $\widetilde{SL(2, \R)}$. The calculation of the WRT
invariant in terms of false theta functions proceeds as in the earlier
examples.
Let $H = \lvert H_1(M, \Z) \rvert = 2p + 1$, $P = p (2p + 1)$, and $\Delta_p =
\phi/4 - 1/2$. Also define three integers $u = P - 4p - 1$, $v = P - 2p - 1$,
and $w = P - 1$ that lie between $0$ and $P$. We have the WRT invariant
\begin{align}
    \xi^{\Delta_p} W_M(\xi) & = \frac{1}{2} \sum_{a \in \Z/H\Z} e^{2\pi i
    \frac{r}{s} \left(\frac{s^2 a^2}{H}(p + 1) \right)} \left(\tilde\Psi^{(u) +
    (w)}_{P}(s/r) - 2\xi^{ra/H} \tilde\Psi^{(v)}_P(s/r)\right) \\
    & = \sum_{a \in \pi_0 (\mathcal M\ped{flat}\ap{ab}(M, SL(2, \C)))} e^{2\pi i \frac{r}{s}
    \CS[a]} \xi^{\Delta_p} W^{a}_M(\xi),
\end{align}
where $\pi_0 (\mathcal M\ped{flat}\ap{ab}(M, SL(2, \C))) \cong \{0, 1, \ldots,
p\}$, and we can read off the terms
\begin{equation}
    \xi^{\Delta_p} W^{a}_M(\xi) = 
    \begin{cases}
        \displaystyle \frac{1}{2} \tilde\Psi^{(u) - 2(v) + (w)}_P(s/r) &
        \text{if $a = 0$}, \\
        \displaystyle \tilde\Psi^{(u) + (w)}_P(s/r) - 2\cos\frac{2\pi s a}{H}
        \tilde\Psi^{(w)}_P(s/r) & \text{otherwise}.
    \end{cases}
\end{equation}
As the terms $W_M^{a}(\xi)$ are given by false theta functions (up to an
overall fractional power of $\xi$), their asymptotics are determined by the
behavior of false theta functions under a modular $S$-transformation. As
before the resulting terms can be interpreted as expansions around flat
connections. Of particular interest is the term
\begin{align}
    W^{0}_M(\xi) & \simeq -\sqrt{\frac{r}{is}} \xi^{-\Delta_p} \sum_{a = 0}^{P -
    1} (M^u_a - 2M^v_a + M^w_a) \frac{1}{2} \tilde \Psi^{(a)}_P(-r/s) + \sum_{k
    = 0}^\infty \frac{c_k}{k!} \left(\frac{\pi i s}{2Pr}\right)^k \\
    & = e^{2\pi i \frac{r}{s} \left(-\frac{(P - 1)^2}{4P}\right)}
    \tilde\xi^{-\frac{(P - 1)^2}{4P}} \frac{1}{2} \tilde\Psi^{H(u) +
    H(w)}_P(-r/s) \Big[ \cdots \Big] + \text{other flat
    connections},
\end{align}
because for the geometric flat connection $\CS[A_*] = - \chi^2/4e = - (P -
1)^2/4P \bmod 1$. And we can read off
\begin{equation}
    P^{0}_{A_*}(\tilde \xi) = \tilde \xi^{-\frac{(P - 1)^2}{4P}}
    \frac{1}{2} \tilde\Psi^{H(u) + H(v)}_P(-r/s)
\end{equation}
and
\begin{equation}
    I^{0}_{A_*}(s/r) = -\frac{1}{H} \sqrt{\frac{r}{is}} \xi^{-\Delta_p}
    \left(M^u_u(P) -2 M^v_u(P) + M^w_u(P)\right),
\end{equation}
and verify that
\begin{equation}
    P^{0}_{A_*}(\tilde \xi) = \tilde\xi^{\Delta_p - \frac{(P - 1)^2}{4P}} \sum_{a
    \in \pi_0 (\mathcal M\ped{flat}\ap{ab}(M, SL(2, \C)))} W^{a}_M(\xi),
\end{equation}
as in Conjecture~\ref{conj:qhs}.

We also refer to \cite{NPQ2025} for the analysis of the asymptotic expansion of WRT invariant at general roots of unity for some other manifolds.
 
\section*{Acknowledgments}

P.P. would like to thank Stavros Garoufalidis, Sergei Gukov, Thomas Nicosanti, Du Pei, Josef Svoboda, Johann Quenta for
discussions on the topic. A.S. would like to thank Asem Abdelraouf for
discussions about quadratic Gauss sums.

\appendix

\section{Gauss sum miscellanea}

\label{app:gauss-sums}

For a positive integer $r$, and an integer $s$, the \emph{quadratic Gauss sum}
is defined as
\begin{equation}
    G(s, r) \coloneqq \sum_{n \in \Z/r\Z} e^{\frac{2\pi i s n^2}{r}}.
\end{equation}
When $s$ and $r$ are coprime, the Gauss sum $G(s, r)$ can be evaluated
explicitly \cite{Murty17},
\begin{equation}
    G(s, r) = 
    \begin{cases}
        \displaystyle \left(\frac{r}{s}\right) (1 + i^{s}) \sqrt{r} & r = 0 \bmod 4, \\
        \displaystyle \left(\frac{s}{r}\right) \sqrt{r} & r = 1 \bmod 4, \\
        \;0 \phantom{\displaystyle\frac{1}{2}} & r = 2 \bmod 4, \\
        \displaystyle \left(\frac{s}{r}\right) i \sqrt{r} & r = 3 \bmod 4,
    \end{cases}
\end{equation}
where $\left(\frac{\cdot}{\cdot}\right)$ is the \emph{Jacobi symbol}.
If $s$ and $r$ are not coprime, let $g = \gcd(s, r)$ and define $s_1 = s/g$,
$r_1 = r/g$ so that $s_1$ and $r_1$ are coprime, and a short calculation shows
that $G(s, r) = g\, G(s_1, r_1)$.

We shall also come across quadratic Gauss sums with a linear piece in its
exponent like
\begin{equation}
    \sum_{n \in \Z/r\Z} \xi^{Pn^2 + 2An},
\end{equation}
where $\xi = e^{\frac{2\pi i s}{r}}$ is a primitive $r$-th root of unity and $P$, $A$
are integers. Such sums can be evaluated with a square completion argument as
follows. Suppose that $P$ is coprime with $r$ (we shall relax
this condition in the next paragraph) and that $r$ is odd. Since we are working
modulo $r$ in the exponent to the root of unity $\xi$, we have
\begin{equation}
    Pn^2 + An = P(n + 2^* P^* A)^2 - 4^* P^* A^2,
\end{equation}
where $^*$ denotes inversion in $\Z/r\Z$, so that
\begin{equation}
    \sum_{n \in \Z/r\Z} \xi^{Pn^2 + 2An} = \xi^{-A^2/4P} \sum_{n \in \Z/r\Z}
    \xi^{P(n + 2^*P^* A)^2} = \xi^{-A^2/4P} G(sP, r),
\end{equation}
where $-A^2/4P$ in the exponent should be understood as an element of $\Z/r\Z$.

If $P$ is not coprime with $r$ (and therefore not invertible in $\Z/r\Z$),
let $g = \gcd(r, P)$, define $r_1 = r/g$, $P_1 = P/g$. And instead of summing
over $n \in \Z/r\Z$, let $n = r_1j + k$ and sum over $j = 0, \ldots, g - 1$ and
$k = 0, \ldots, r_1 - 1$ for
\begin{equation}
    \sum_{n = 0}^{r - 1} \xi^{Pn^2 + An} = \sum_{j = 0}^{g - 1} \xi^{r_1 Aj}
    \sum_{k = 0}^{r_1 - 1} \xi_1^{P_1 k^2 + A_1 k} = g \delta_{g \mid A}\,
    \xi_1^{-A_1^2/4P_1} G(sP_1, r_1).
\end{equation}
where the sum over $j$ is a sum of $g$-th roots of unity $\xi^{r_1 A}$, and is
zero unless $g$ divides $A$, and the sum over $k$ is like the one we started
out with, in which $r$, $P$, and $A$ are replaced by $r_1$, $P_1$, and $A_1$
respectively. Note that the expression on the right-hand side vanishes unless
$A_1$ is an integer, so we can use the square completion argument as before.

Using the above formulas, we have
\begin{equation}
    F(U^{\pm 1}, q) = \mp \frac{q^{\mp 3/4}}{q^{1/2} - q^{-1/2}}
    \frac{G_\pm(q^{1/4})}{2} = \mp \left(\frac{r}{s}\right) 
    \left(\frac{1 \pm i^s}{\sqrt 2}\right)
    \sqrt{2r} \frac{q^{\mp 3/4}}{q^{1/2} - q^{-1/2}}.
\end{equation}
To see why the first equality is true, consider the Gauss sum involved in
$F(U^{\pm 1}, q)$, for $\sigma = \pm 1$, 
\begin{align}
    \sum_{n = 1}^{r - 1} q^{\sigma (n^2 - 1) / 4} [n]^2
    & = \frac{q^{-\sigma/4}}{(q^{1/2} - q^{-1/2})^2} \sum_{n = 1}^{r - 1}
    q^{\sigma n^2/4} (q^{n} + q^{-n} - 2) \\
    & = \frac{q^{-\sigma/4}}{(q^{1/2} - q^{-1/2})^2} \sum_{n \in \Z/2r\Z}
    q^{\sigma n^2/4} (q^n - 1) \\
    & = \frac{q^{-\sigma/4}}{(q^{1/2} - q^{-1/2})^2} (q^{-\sigma} - 1) 
    \sum_{n \in \Z/2r\Z} q^{\sigma n^2/4} \\
    & = -\sigma\frac{q^{-3\sigma/4}}{q^{1/2} - q^{-1/2}} \sum_{n \in \Z/2r\Z}
    q^{\sigma n^2/4}.
\end{align}
With this, the normalization in (\ref{eqn:wrt-def}) can be written as
\begin{equation}
    F(U^{\pm 1}, q)^{b_+} F(U^{\pm 1}, q)^{b_-} = (-1)^{b_+}
    \left(\frac{r}{s}\right)^N \left(\frac{1 + i^s}{\sqrt 2}\right)^{\sigma(B)}
    \frac{(2r)^{N/2} q^{-3\sigma(B)/4}}{(q^{1/2} - q^{1/2})^N},
\end{equation}
where $\sigma(B) = b_+ - b_-$ is the signature of the linking matrix.

\begin{theorem}[Deloup and Turaev \cite{deloup2007reciprocity}]
    \label{thm:reciprocity}
    Let $\langle\cdot,\cdot\rangle$ be an inner product on $\R^N$, 
    $B : \R^N \to \R^N$ be a self-adjoint automorphism, $\psi \in \R^N$, and $r$
    a positive integer. Assume that $B(\Z^N) \subseteq \Z^N$ and
    \begin{equation}
        \frac{r}{2}\langle x, Bx\rangle,\ 
        \langle x, Bx^\prime\rangle,\ 
        \langle x, rx^\prime\rangle,\ 
        r \langle x, \psi \rangle \in \Z 
        \quad \text{for all $x, x^\prime \in \Z^N$}.
    \end{equation}
    Then 
    \begin{align}
        \sum_{x \in \Z^N / r\Z^N}
        \exp & \left( \frac{\pi i}{r} \langle x, Bx \rangle 
        + 2\pi i\langle x, \psi \rangle \right) \nonumber \\
        & =
        \frac{e^{\pi i \sigma / 4} r^{N/2}}{\lvert\det B \rvert^{1/2}} \sum_{y \in
        \Z^N / B\Z^N} \exp\left(- \pi i r \langle y + \psi,
        B^{-1}(y + \psi)\rangle \right),
    \end{align}
    for arbitrary lifts of $x, y$ to $\Z^N$.
\end{theorem}

\section{Integrality of false theta limits (proof of Proposition~\ref{prop:int-coeff})}

\label{app:int-coeff}

\begin{proposition*}
    \intcoeffprop
\end{proposition*}

Recall the definition of $\tilde\Phi^{\vec a}_{\vec p}(s/r)$
\begin{equation}
    \tilde\Phi^{\vec a}_{\vec p}(s/r) = \frac{1}{2} \sum_{l = 1}^{2Pr}
    \phi^{\vec a}_{\vec p}(l) \left(1 - \frac{l}{Pr}\right) \xi^{l^2/4P},
\end{equation}
where the $2P$-periodic function $\phi^{\vec a}_{\vec p}$ is defined by
\begin{equation}
    \phi^{\vec a}_{\vec p}(l) = 
    \begin{cases}
        -\epsilon_1 \epsilon_2 \epsilon_3 & \text{if}\ l = P\left(1 + \sum_{j =
        1}^3 \frac{\epsilon_j a_j}{p_j}\right) \bmod 2P, \\
        0 & \text{otherwise},
    \end{cases}
\end{equation}
for $\epsilon_j = \pm 1$. With the antisymmetry property $\phi^{\vec a}_{\vec
p}(-l) = -\phi^{\vec a}_{\vec p}(l)$, we have
\begin{equation}
    \sum_{l = 1}^{2Pr} \phi^{\vec a}_{\vec p}(l) \xi^{l^2/4P} = 0,
\end{equation}
so that
\begin{equation}
    \frac{1}{2} \tilde\Phi^{\vec a}_{\vec p}(s/r) = \frac{1}{4Pr} \sum_{l = 1}^{2Pr}
    \phi^{\vec a}_{\vec p}(l) l \xi^{l^2/4P}.
\end{equation}

Note that all nonzero summands in the above sum have the same power of $\xi$
modulo 1 because for $l = P\left(1 + \sum_j \frac{\epsilon_j a_j}{p_j}\right)
\bmod 2P$ we have
\begin{equation}
    \frac{l^2}{4P} = \frac{P}{4} \left(1 + \sum_j \frac{\epsilon_j
    a_j}{p_j}\right)^2 = \frac{P}{4}\left(1 + \sum_j \frac{a_j}{p_j}\right)^2 =
    -\CS[\vec a] \mod 1, 
\end{equation}
where $\CS[\vec a]$ is the Chern--Simons action of the flat connection on
$\Sigma(p_1, p_2, p_3)$ that is denoted by rotation numbers $(a_1, a_2, a_3)$.

Next, we want to use the $2P$-periodicity of $\phi^{\vec a}_{\vec p}$ to write
the sum over $l$ above as a sum over $L = 0, \ldots, r - 1$ and $\epsilon \in
\{\pm 1\}^3$. To do so, we make the replacements $l \to 2PL + P + P\sum_j
\frac{\epsilon_j a_j}{p_j}$ and $\phi^{\vec a}_{\vec p}(l) \to -\epsilon_1
\epsilon_2 \epsilon_3$, and observe that the sums
\begin{equation}
    \frac{1}{4Pr} \sum_{l = 0}^{2Pr - 1} \phi^{\vec a}_{\vec p}(l) l
    \xi^{l^2/4P} 
\end{equation}
and
\begin{equation}
    - \frac{1}{4Pr} \sum_{L = 0}^{r - 1} \sum_{\epsilon \in \{\pm 1\}^3}
    \epsilon_1 \epsilon_2 \epsilon_3 \left(2PL + P + P\sum_{j=1}^3
    \frac{\epsilon_j a_j}{p_j}\right) \xi^{\frac{P}{4}\left(1 + \sum_j
    \frac{\epsilon_j a_j}{p_j}\right)^2},
\end{equation}
though not quite equal, differ by a polynomial in $\xi$ (up to an overall
fractional power) with integer coefficients. So, to prove the proposition, it
will suffice to show that the second sum over $L = 1, \ldots, r - 1$ and $\epsilon
\in \{\pm 1\}^3$ is an element of $\xi^{-\CS[\vec a]} \Z[\xi]$.

It will be convenient to get rid of this fractional part of the exponent, and
we will do so by multiplying $\xi^{-l_0^2/4P}$ with $l_0 = P\left(1 - \sum_j
\frac{a_j}{p_j}\right)$. Define the quadratic function $F_\epsilon$ by
\begin{equation}
    \frac{l^2 - l_0^2}{4P} = P\left(L + \sum_{j = 1}^3 \frac{\epsilon_j +
    1}{2}\frac{a_j}{p_j}\right) \left(L + 1 + \sum_{j = 1}^3 \frac{\epsilon_j -
    1}{2}\frac{a_j}{p_j}\right) \eqqcolon F_\epsilon(L),
\end{equation}
and write the sum of interest, after removing the fraction exponent, as
\begin{equation}
    \frac{1}{4Pr} \sum_{L = 0}^{r - 1} \sum_{\epsilon \in \{\pm 1\}^3}
    \epsilon_1 \epsilon_2 \epsilon_3 \left(2PL + P + P\sum_{j = 1}^3
    \frac{\epsilon_j a_j}{p_j}\right) \xi^{F_\epsilon(L)}.
\end{equation}

We will deal with each term in the parentheses above one by one. We start with
the second term.

\begin{lemma}
    \label{lem:prop-2-term-2}
    With $F_\epsilon$ as defined above and $\xi$ an $r$-th root of unity, we have
    \begin{equation}
        \sum_{\epsilon \in \{\pm 1\}^3} \epsilon_1 \epsilon_2 \epsilon_3
        \sum_{L = 0}^{r - 1} \xi^{F_\epsilon(L)} = 0.
    \end{equation}
\end{lemma}
\begin{proof}
    We use the fact that $F_\epsilon$ is a well-defined quadratic form on
    $\Z/r\Z$ to write the above sum as
    \begin{equation}
        \sum_{\epsilon \in \{\pm 1\}^3} \epsilon_1 \epsilon_2 \epsilon_3
        \sum_{L \in \Z/r\Z} \xi^{F_\epsilon(L)} = 
        \sum_{\epsilon \in E} \epsilon_1\epsilon_2\epsilon_3 \sum_{L \in
        \Z/r\Z} \xi^{F_\epsilon(L)}
        - \sum_{\epsilon \in E} \epsilon_1\epsilon_2\epsilon_3 \sum_{L \in
        \Z/r\Z} \xi^{F_{-\epsilon}(L)},
    \end{equation}
    where $E = \{(1, 1, 1), (1, 1, -1), (1, -1, 1), (-1, 1, 1)\}$, and conclude
    that the right-hand side is zero because $F_{-\epsilon}(L) = F_\epsilon(-L
    - 1) \bmod r$.
\end{proof}

To deal with the other terms, we need to understand how the sums $\sum_{L \in
\Z/r\Z} \xi^{F_\epsilon(L)}$ for different $\epsilon$'s are related to each
other.
A short calculation gives the coefficients of the quadratic polynomial
$F_\epsilon(L) = PL^2 + A_\epsilon L + B_\epsilon$,
\begin{equation}
    A_\epsilon = P\left(1 + \sum_{j = 1}^3 \frac{\epsilon_j a_j}{p_j}\right)
    \ \text{and}\ 
    B_\epsilon = P \sum_{j = 1}^3 \frac{\epsilon_j + 1}{2} \frac{a_j}{p_j} + P
    \sum_{i, j} \frac{\epsilon_i + 1}{2} \frac{\epsilon_j - 1}{2} \frac{a_i
    a_j}{p_i p_j}.
\end{equation}
Like in Appendix~\ref{app:gauss-sums}, let $g = \gcd(r, P)$ and define $r' =
r/g$, $P' = P/g$, and let $g_j = \gcd(r, p_j)$ and define $p'_j = p_j/g_j$ for
$j = 1, 2, 3$.
Instead of summing over $L = 0, \ldots, r - 1$, write
$L = r' k + l$ and sum over $k = 0, \ldots, g - 1$ and $l = 0, \ldots, r' - 1$
to write
\begin{equation}
    \sum_{L \in \Z/r\Z} \xi^{F_\epsilon(L)} = \sum_{k = 0}^{g - 1} \xi^{r'
    A_\epsilon k}\sum_{l = 0}^{r' - 1} \xi^{F_\epsilon(l)} = g \delta_{g \mid
    A_\epsilon} \sum_{l \in \Z/r' \Z} \xi'^{F'_\epsilon(l)},
\end{equation}
where $\xi' = \xi^g$ as before and $F'_\epsilon = F_\epsilon / g$. To see why
the last sum over $l$ can be written as a sum on $\Z/r'\Z$, note that the
right-hand side is zero unless $g$ divides $A_\epsilon$, and from the
expressions of $A_\epsilon$ and $B_\epsilon$ above and using the fact that
$p_j$'s are mutually coprime, one can deduce that $g \mid A_\epsilon$ implies
$g \mid B_\epsilon$.

Indeed, it is true that $g \mid A_\epsilon$ if and only if $g_j \mid a_j$ for
every $j = 1, 2, 3$. If this is the case, common factors in $F'_\epsilon$ can
be canceled and
\begin{equation}
    F'_\epsilon(l) = P'\left(l + \sum_{j = 1}^3 \frac{\epsilon_j + 1}{2}
    \frac{a'_j}{p'_j}\right)\left(l + 1 + \sum_{j = 1}^3 \frac{\epsilon_j - 1}{2}
    \frac{a'_j}{p'_j}\right),
\end{equation}
where $a'_j = a_j/g_j$ for $j = 1, 2, 3$. We have the following relations
(modulo $r'$) between $F'_\epsilon$ for different $\epsilon$'s,
\begin{align}
    \label{eqn:Fepsilon-rels-1}
    F'_{-\epsilon}(l) & = F'_{\epsilon}(-l - 1), \\
    F'_{-++}(l) & = F'_{+++}(l - a'_1 p^{\prime *}_1), \\
    F'_{+-+}(l) & = F'_{+++}(l - a'_2 p^{\prime *}_2), \\
    \label{eqn:Fepsilon-rels-2}
    F'_{++-}(l) & = F'_{+++}(l - a'_3 p^{\prime *}_3),
\end{align}
which let us relate all $F'_\epsilon$'s to each other by an appropriate shift of
the argument. In particular,
\begin{equation}
    \label{eqn:Fepsilon-sum-equal}
    \sum_{l \in \Z/r'\Z} \xi'^{F'_{\epsilon}(l)}
    = \sum_{l \in \Z/r'\Z} \xi'^{F'_{\varepsilon}(l)}
\end{equation}
for any $\epsilon, \varepsilon \in \{-1, 1\}^3$.

With these properties of the quadratic form $F_\epsilon$ understood, the
following lemma---which implies that the third term is zero---can be proved.

\begin{lemma}
    \label{lem:prop-2-term-3}
    With $F_\epsilon$ as defined above and $\xi$ an $r$-th root of unity, we
    have
    \begin{equation}
        \sum_{\epsilon \in \{\pm 1\}^3} \epsilon_1 \epsilon_2 \epsilon_3
        \sum_{L = 0}^{r - 1} \epsilon_j \xi^{F_\epsilon(L)} = 0,
        \quad\text{for every $j = 1, 2, 3$}.
    \end{equation}
\end{lemma}

\begin{proof}
    Here we will show what happens for $j = 1$ (other cases are completely
    analogous). As in the above discussion, we decompose the sum over $L$ to
    write
    \begin{equation}
        \sum_{\epsilon \in \{\pm\}^3} \epsilon_1 \epsilon_2 \epsilon_3 \sum_{L
        = 0}^{r - 1} \epsilon_1 \xi^{F_\epsilon(L)} = g \delta_{g_j \mid a_j,\,
        j = 1, 2, 3} \sum_{\epsilon \in \{\pm 1\}^3} \epsilon_2 \epsilon_3
        \sum_{l \in \Z/r'\Z} \xi'^{F'_\epsilon(l)}.
    \end{equation}
    Next, using the relations between $F'_\epsilon$'s and the fact that
    $\sum_{\epsilon \in \{\pm 1\}^3} \epsilon_2 \epsilon_3 = 0$, we have
    \begin{equation}
        \sum_{\epsilon \in \{\pm 1\}^3} \epsilon_2 \epsilon_3
        \sum_{l \in \Z/r'\Z} \xi'^{F'_\epsilon(l)} = \sum_{\epsilon \in
        \{\pm 1\}^3} \epsilon_2 \epsilon_3 \sum_{l \in \Z/r'\Z}
        \xi'^{F'_{+++}(l)} = 0.
    \end{equation}
\end{proof}

Finally, for the first term, we have the following.

\begin{lemma}
    With $F_\epsilon$ as defined above and $\xi$ a root of unity of odd order
    $r$, the polynomial
    \begin{equation}
        \frac{1}{2r} \sum_{\epsilon \in \{\pm 1\}^3} \epsilon_1 \epsilon_2
        \epsilon_3 \sum_{L = 0}^{r - 1} L \xi^{F_\epsilon(L)},
    \end{equation}
    in $\xi$ has integer coefficients.
\end{lemma}

\begin{proof}

First, we get rid of the factor of half by noting that
\begin{align}
    \frac{1}{2} \sum_{\epsilon \in \{\pm 1\}^3} \epsilon_1 \epsilon_2 \epsilon_3
    \sum_{L = 0}^{r - 1} L \xi^{F_{\epsilon}(L)}
    & = \frac{1}{2} \sum_{\epsilon \in E} \epsilon_1 \epsilon_2 \epsilon_3
    \sum_{L = 0}^{r - 1} L \xi^{F_{\epsilon}(L)}
    - \frac{1}{2} \sum_{\epsilon \in E} \epsilon_1 \epsilon_2 \epsilon_3
    \sum_{L = 0}^{r - 1} L \xi^{F_{-\epsilon}(L)} \\
    & = \sum_{\epsilon \in E} \epsilon_1 \epsilon_2 \epsilon_3 \sum_{L = 0}^{r
    - 1} L \xi^{F_\epsilon(L)} - \frac{r - 1}{2} \sum_{\epsilon \in E}
    \epsilon_1 \epsilon_2 \epsilon_3 \sum_{L \in \Z/r\Z} \xi^{F_\epsilon(L)} \\
    & = \sum_{\epsilon \in E} \epsilon_1 \epsilon_2 \epsilon_3 \sum_{L = 0}^{r
    - 1} L \xi^{F_\epsilon(L)},
\end{align}
where $E = \{(1, 1, 1), (1, 1, -1), (1, -1, -1), (1, -1, 1)\}$. We can show
\begin{equation}
    \sum_{\epsilon \in E} \epsilon_1 \epsilon_2 \epsilon_3
    \sum_{L \in \Z/r\Z} \xi^{F_\epsilon(L)} = g \delta_{g_j \mid a_j,\, j = 1,
    2, 3} \sum_{\epsilon \in E} \epsilon_1 \epsilon_2 \epsilon_3 \sum_{l \in
    \Z/r'\Z} \xi'^{F'_{+++}(l)} = 0,
\end{equation}
by using relations among different $F'_\epsilon$'s and $\sum_{\epsilon \in
E} \epsilon_1 \epsilon_2 \epsilon_3 = 0$. Note that the choice of the index set
$E$ above is not unique, and we can pick any $E \subset \{-1, 1\}^3$ such that
$E \cap -E = \varnothing$ and $E \cup -E = \{-1, 1\}^3$.

Next, we decompose the sum over $L$ as before
\begin{align}
    \frac{1}{r} \sum_{L = 0}^{r - 1} L \xi^{F_\epsilon(L)} & = \frac{1}{r} \sum_{k = 0}^{g
    - 1} \sum_{l = 0}^{r' - 1} (r'k + l) \xi^{F_\epsilon(r'k + l)} \\
    & = \frac{1}{g} \sum_{k = 0}^{g - 1} k
    \xi^{r'A_\epsilon k} \sum_{l = 0}^{r' - 1} \xi^{F_\epsilon(l)} + 
    \delta_{g\mid A_\epsilon} \frac{1}{r'} \sum_{l = 0}^{r' - 1} l \xi'^{F'_\epsilon(l)}.
    \label{eqn:prop-2-term-1-decomposed}
\end{align}

Since $g = g_1 g_2 g_3$ and $g_j$'s are mutually coprime, we can decompose the
sum over $k$ in the first term in the following symmetric way, $k \to
g \sum_{j=1}^3 \frac{k_j}{g_j}$, where $k_j$'s all run from $0$ to $g_j - 1$,
\begin{equation}
    \frac{1}{g} \sum_{k = 0}^{g - 1} k \xi^{r' A_\epsilon k}
    = \sum_{k_1 = 0}^{g_1 - 1} \sum_{k_2 = 0}^{g_2 - 1} \sum_{k_3 = 0}^{g_3 - 1}
    \left(\sum_{j = 1}^3 \frac{k_j}{g_j}\right) \xi^{r \frac{\epsilon_1 a_1 k_1}{g_1} +
    r \frac{\epsilon_2 a_2 k_2}{g_2} + r \frac{\epsilon_3 a_3 k_3}{g_3}}.
\end{equation}

We will work out the term $j = 1$ in detail; terms $j = 2, 3$ are
completely analogous. We have
\begin{equation}
    \frac{1}{g_1} \sum_{k_1 = 0}^{g_1 - 1} \sum_{k_2 = 0}^{g_2 - 1} \sum_{k_3 =
    0}^{g_3 - 1} k_1 \xi^{r
    \frac{\epsilon_1 a_1 k_1}{g_1} + r \frac{\epsilon_2 a_2 k_2}{g_2} + r
    \frac{\epsilon_3 a_3 k_3}{g_3}} = -g_2 g_3 \delta_{g_2 \mid a_2} \delta_{g_3 \mid
    a_3} \frac{1}{1 - \xi^{r\frac{\epsilon_1 a_1}{g_1}}},
\end{equation}
where we used the identity for the sum of roots of unity for sums over $k_2$
and $k_3$, and the identity
\begin{equation}
    \sum_{n = 0}^{h - 1} n \chi^n = -\frac{h}{1 - \chi},
\end{equation}
where $\chi$ is an $h$-th root of unity, for the sum over $k_1$. With the sums
over $\epsilon$ and $l$, this term becomes
\begin{equation}
    \label{eqn:prop-2-divisible}
    -g_2 g_3 \delta_{g_j \mid a_j,\, j = 2, 3} \frac{1}{1 -
    \xi^{r\frac{a_1}{g_1}}} \sum_{\epsilon \in E} \epsilon_2 \epsilon_3 \sum_{l
    = 0}^{r' - 1} \xi^{F_\epsilon(l)},
\end{equation}
where we have chosen $E \subset \{-1, 1\}^3$ such that $\epsilon_1 = 1$ for
every $\epsilon \in E$. To complete the proof, it suffices to show that
$\sum_{\epsilon \in E} \sum_{l = 0}^{r' - 1} \xi^{F_\epsilon(l)}$ is divisible
by $1 - \xi^{r a_1 / g_1}$ in $\Z[\xi]$.
Since the term above is nonzero even when $g_1$ does not divide $a_1$, we
cannot cancel factors in $F_\epsilon$, and $F'_\epsilon$ might not be
well-defined on $\Z/r'\Z$. But we still have 
\begin{equation}
    \label{eqn:Fepsilon-r'-shift}
    F_\epsilon(l + r'k) = F_\epsilon(l) + r'A_\epsilon k,
\end{equation}
and the relations,
\begin{align}
    \label{eqn:Fepsilon-r'-rels-1}
    F_{++-}(l) & = F_{+++}(l - a_3'p_3'^*), \\
    \label{eqn:Fepsilon-r'-rels-2}
    F_{+--}(l) & = F_{+-+}(l - a_3'p_3'^*)
\end{align}
modulo $r$, where $0 < p_3'^* < r$ is an integer such that $p_3' p_3'^* = 1 \bmod
r$.

To show that the sum in (\ref{eqn:prop-2-divisible}) is divisible by $1 -
\xi^{ra_1/g_1}$, we use (\ref{eqn:Fepsilon-r'-rels-1}) for
\begin{equation}
    \sum_{l = 0}^{r' - 1} \xi^{F_{+++}(l)}
    - \sum_{l = 0}^{r' - 1} \xi^{F_{++-}(l)} =
    \sum_{l = 0}^{r' - 1} \xi^{F_{+++}(l)}
    - \sum_{l = -a_3'p_3'^*}^{-a_3'p_3'^* + r' - 1} \xi^{F_{+++}(l)},
\end{equation}
and note that $1 - \xi^{r'A_{+++}}$ is a factor of the above because for every
$l' \in \{-a_3'p_3'^*, -a_3'p_3'^* + 1, \ldots, -a_3'p_3'^* + r' - 1\}$, there
exists an $l \in \{0, 1, \ldots, r' - 1\}$ such that $l' = l + r'k$ for some
integer $k$, and due to (\ref{eqn:Fepsilon-r'-shift}) $F_{+++}(l') -
F_{+++}(l) = r'A_{+++} k$. Similarly, we can use (\ref{eqn:Fepsilon-r'-rels-2})
and (\ref{eqn:Fepsilon-r'-shift}) to see that $1 - \xi^{r'A_{+++}}$ is also a
factor of $\sum_{l = 0}^{r' - 1} \xi^{F_{+--}(l)} - \sum_{l = 0}^{r' - 1}
\xi^{F_{+-+}(l)}$, so that $1 - \xi^{r'A_{+++}}$ is a factor of the sum in
(\ref{eqn:prop-2-divisible}). But since
\begin{equation}
    \frac{1 - \xi^{r'A_{+++}}}{1 - \xi^{r\frac{a_1}{g_1}}} = 
    \frac{1 - \xi^{r \frac{a_1 p_2' p_3'}{g_1}}}{1 - \xi^{r\frac{a_1}{g_1}}} =
    1 + \xi^{r\frac{a_1}{g_1}} + \cdots + \xi^{r\frac{a_1 p_2'p_3'}{g_1}} 
    \in \Z[\xi],
\end{equation}
the term (\ref{eqn:prop-2-divisible}) is an element of $\Z[\xi]$.

In the second term (\ref{eqn:prop-2-term-1-decomposed}), we can assume $g_j \mid
a_j$ for all $j = 1, 2, 3$ because it vanishes unless $g$ divides $A_\epsilon$,
and therefore $F'_\epsilon$ is a well-defined quadratic form on $\Z/r'\Z$. If $2$
is invertible in $\Z/r'\Z$, then we can write this quadratic polynomial in the
following diagonal form $F'_\epsilon(l) = P'(l - D_\epsilon)^2 + C_\epsilon$
modulo $r'$, for some integers $0 \leq C_\epsilon, D_\epsilon < r'$. Using the
relations (\ref{eqn:Fepsilon-rels-1}--\ref{eqn:Fepsilon-rels-2}) among different
$F'_\epsilon$'s, we can deduce $C_\epsilon = C_\varepsilon$ for any $\epsilon,
\varepsilon \in \{-1, 1\}^3$, and the following relations among $D_\epsilon$'s,
\begin{align}
    \label{eqn:Depsilon-rels-1}
    D_{-\epsilon} & = 1 - D_{\epsilon}, \\
    D_{-++} & = D_{+++} + a'_1 p'^*_1, \\
    D_{+-+} & = D_{+++} + a'_2 p'^*_2, \\
    \label{eqn:Depsilon-rels-2}
    D_{++-} & = D_{+++} + a'_3 p'^*_3.
\end{align}

With the diagonal form of $F'_\epsilon$ as above, we write
\begin{equation}
    \frac{1}{r'}\sum_{l = 0}^{r' - 1} l \xi'^{F'_\epsilon(l)}
    = \frac{1}{r'}\sum_{l = 0}^{r' - 1} l \xi'^{P'(l - D_\epsilon)^2 + C_\epsilon}
    = \frac{1}{r'}\sum_{l = -D_\epsilon}^{-D_\epsilon + r' - 1} l \xi'^{P'l^2 +
    C_\epsilon} + \frac{D_\epsilon}{r'} \sum_{l = 0}^{r' - 1}
    \xi'^{F'_\epsilon(l)}.
\end{equation}
But we can show
\begin{equation}
    \sum_{\epsilon \in E} \epsilon_1 \epsilon_2 \epsilon_3 D_\epsilon \sum_{l =
    0}^{r' - 1} \xi'^{F'_\epsilon(l)} = 
    \sum_{\epsilon \in E} \epsilon_1 \epsilon_2 \epsilon_3 D_\epsilon \sum_{l =
    0}^{r' - 1} \xi'^{F'_{+++}(l)} = 0,
\end{equation}
using (\ref{eqn:Fepsilon-sum-equal}), and $\sum_{\epsilon \in E} \epsilon_1
\epsilon_2 \epsilon_3 D_\epsilon = 0$, which follows from the relations
(\ref{eqn:Depsilon-rels-1}--\ref{eqn:Depsilon-rels-2}) between $D_\epsilon$'s. To
deal with the other term, we use
\begin{equation}
    \sum_{l = -r + 1}^{r - 1} l \xi'^{P'l^2 + C_\epsilon} = 0,
\end{equation}
to write
\begin{align}
    - \sum_{l = -D_\epsilon}^{-D_\epsilon + r' - 1} l \xi'^{P'l^2 + C_\epsilon} & =
    \sum_{l = -r' + 1}^{-D_\epsilon - 1}  l \xi'^{P'l^2 + C_\epsilon}
    + \sum_{l = -D_\epsilon + r'}^{r' - 1} l \xi'^{P'l^2 + C_\epsilon} \\
    & = \sum_{l = 1}^{-D_\epsilon + r' - 1} (l - r') \xi'^{P'(l - r')^2 +
    C_\epsilon} + \sum_{l = -D_\epsilon + r'}^{r' - 1} l \xi'^{P'l^2 + C_\epsilon}
    \\
    & = \sum_{l = 1}^{r' - 1} l \xi'^{P'l^2 + C_\epsilon} - \sum_{l =
    1}^{-D_\epsilon + r' - 1} r' \xi'^{P'l^2 + C_\epsilon},
\end{align}
so that
\begin{equation}
    \frac{1}{r'} \sum_{l = -D_\epsilon}^{-D_\epsilon + r - 1} l \xi'^{P'l^2 +
    C_\epsilon} = \sum_{l = 1}^{-D_\epsilon + r - 1} \xi^{Pl^2 + gC_{+++}} -
    \frac{1}{r'} \sum_{l = 0}^{r' - 1} l\xi'^{P'l^2 + C_{+++}},
\end{equation} 
where we have used that $C_\epsilon$'s are the same for every $\epsilon$. We notice
that the first term is an element of $\Z[\xi]$ and since the second term is
independent of $\epsilon$, we have
\begin{equation}
    \sum_{\epsilon \in E} \epsilon_1 \epsilon_2 \epsilon_3 \sum_{l = 0}^{r' - 1}
    l\xi'^{P'l^2 + C_{+++}} = 0,
\end{equation}
because $\sum_{\epsilon \in E} \epsilon_1 \epsilon_2 \epsilon_3 = 0$.
\end{proof}

\section{\texorpdfstring{$\hat Z$}{Z-hat} and WRT invariants at generic roots of unity}
\label{app:Zhat}

In this section, we shall comment on the conjectural relationship between WRT
invariants at generic roots of unity and radial limits of $\hat Z$-invariants.
The following result about a higher rank version of quadratic Gauss sums will
be important in simplifying the resulting relation.

\begin{proposition}
    \label{prop:high-rk-gauss-sum}
    For any positive integer $N$, let $B : \Z^N \to \Z^N$ be a nondegenerate,
    symmetric, $\Z$-linear map, and let $r$ be an odd positive integer, coprime
    with $\det B$. Then
    \begin{equation}
        \sum_{x \in \Z^N/r\Z^N} e^{\frac{2\pi i}{r} \inner{x}{Bx}} = 
        \begin{cases}
            \displaystyle \left(\frac{\det B}{r}\right) r^{N/2} & r = 1 \bmod 4, \\
            \displaystyle \left(\frac{\det B}{r}\right) i^N r^{N/2} & r = 3 \bmod 4.
        \end{cases}
    \end{equation}
\end{proposition}
\begin{proof}
    Let $r = p_1^{n_1} p_2^{n_2} \cdots p_k^{n_k}$ be the prime factorization
    of $r$ where $p_j$ are distinct odd primes and $n_j \in \Z_{>0}$ for all $j =
    1, 2, \cdots, k$. We will prove this result by induction on $k$. For $k =
    1$, we use a diagonalization result for symmetric, invertible bilinear
    forms over local rings in which $2$ is a unit
    \cite[Corollary~3.4]{milnor1973symmetric}. The matrix $B$ is invertible
    over $\Z/p_1^{n_1}\Z$ because $\det B$ is coprime with $r$, and $2$ is a unit in
    $\Z/p_1^{n_1}\Z$ because $r$ is odd. We have $SBS^T = \mathrm{diag}(b_1,
    \ldots, b_N)$ for some orthonormal matrix $S$, and the result follows from
    a short calculation
    \begin{equation}
        \sum_{x \in \Z^N/p_1^{n_1}\Z^N} e^{\frac{2\pi i}{p_1^{n_1}} \langle x, B x\rangle}
        = \prod_{j = 1}^{N} \sum_{x_j \in \Z/p_1^{n_1}\Z} e^{\frac{2\pi i
        b_j}{p_1^{n_1}}}
        = \left(\frac{\det B}{p_1^{n_1}}\right) G(1, p_1^{n_1})^N.
    \end{equation}
    For the induction step, we factor the sum over
    $\Z^N/r\Z^N$ into $\Z^N/r'\Z^N \oplus \Z^N/p_k^{n_k}\Z_N$, where $r' =
    r/p_k^{n_k}$
    \begin{align}
        \sum_{x \in \Z^N/r\Z^N} e^{\frac{2\pi i}{r} \langle x, Bx \rangle}
        & = \sum_{x \in \Z^N/r'\Z^N} e^{\frac{2\pi i}{r'} p_k^{n_k}\langle x, Bx \rangle}
        \sum_{y \in \Z^N/p^n\Z^N} e^{\frac{2\pi i}{p_k^{n_k}} r' \langle y, By
        \rangle} \\
        & = \left(\frac{\det B}{r}\right) \left[\left(\frac{p_k^{n_k}}{r'}\right)
        \left(\frac{r'}{p_k^{n_k}}\right) \right]^N G(1, r')^N G(1,
        p_k^{n_k})^N.
    \end{align}
    To see that the sign of the Gauss sum is correct, we need to use the
    reciprocity of Jacobi symbols
    \begin{equation}
        \left(\frac{r}{s}\right) \left(\frac{s}{r}\right) = (-1)^{\frac{r -
        1}{2} \frac{s - 1}{2}}.
    \end{equation}
\end{proof}

Suppose that $M$ is a 3-manifold obtained from a genus zero plumbing on a tree graph, and let $B$ be its linking matrix, which we assume to be negative definite\footnote{We expect that the argument below can be in principle generalized to the case of \textit{weakly negative-definite} plumbings. However, this is redundant, since, as was shown in \cite{harichurn2025delta}, any weakly negative-definite rational homology sphere plumbing has a negative-definite representative in its homeomorphism class.}. Let
$\xi = e^{\frac{2\pi i s}{r}}$ be a primitive $r$-th root of unity. With manipulations
similar to \cite[Appendix~A]{Gukov:2017kmk}, we get (see also
\cite{Kucharski:2019fgh,Chung:2019jgw})
\begin{align}
    \tau_M(\xi) = \lim_{q \to \xi} & \frac{1}{2\, (q^{1/2} - q^{-1/2})\, \lvert \det
    B \rvert^{1/2}\, G(r, s)^N} \nonumber \\
    & \phantom{\sum_{a \in \coker s B}} \times \sum_{a \in \coker s B} e^{-2\pi i
    \frac{r}{s} \langle a, B^{-1} a \rangle} \sum_{b \in 2\coker B + \delta}
    e^{-2\pi i \langle a, B^{-1} b \rangle} \hat Z_b(q),
\end{align}
If we let $s$ be coprime with $\det B$, it is possible to factor the sum over
$\coker s B$---which is not invariant under Kirby moves---to sums over
$\Z^N/s\Z^N$ and  $\coker B$ using the isomorphism $\Z^N/s\Z^N \oplus \coker B
\ni (x, a) \mapsto Bx + sa \in \coker sB$. The two sums above become
\begin{equation}
    \sum_{x \in \Z^N/s\Z^N} e^{-2\pi i \frac{r}{s} \langle x, B x \rangle}
    \sum_{a \in \coker B} e^{-2\pi i \frac{r}{s} s^2\langle a, B^{-1} a
    \rangle} \sum_{b \in 2\coker B + \delta} e^{-2\pi i s \langle a, B^{-1} b
    \rangle} \hat Z_b(q).
\end{equation}
We can use formulas from Appendix~\ref{app:gauss-sums} to evaluate the Gauss sums in the
denominator, and if we assume that $s$ is odd, we can use
Proposition~\ref{prop:high-rk-gauss-sum} to evaluate the sum over $x$, so that
\begin{align}
    \left(\frac{\det(-B)}{s}\right) \tau_M(\xi) = \lim_{q \to
    \xi} & \frac{1}{2\, (q^{1/2} - q^{-1/2})\, \lvert \det B \rvert^{1/2}} 
    \nonumber \\
    &
    \times \sum_{a \in \coker B} e^{-2\pi i
    \frac{r}{s} s^2 \langle a, B^{-1} a \rangle} \sum_{b \in 2\coker B + \delta}
    e^{-2\pi i s \langle a, B^{-1} b \rangle} \hat Z_b(q).
\end{align}
Using the assumption that $s$ is coprime with $\lvert\det B \rvert=\lvert
H_1(M)\rvert$ we can rewrite the formula in the form independent of the plumbing representation (cf. \cite{Costantino:2021yfd}):
\begin{align}
    \left(\frac{\lvert H_1(M)\rvert}{s}\right) \tau_M(\xi) = \lim_{q \to
    \xi} & \frac{1}{2\, (q^{1/2} - q^{-1/2})\, \lvert H_1(M) \rvert^{1/2}} \times
    \nonumber \\
    & \phantom{\sum} 
    \sum_{a \in H_1(M)} e^{-2\pi i
    \frac{r}{s}\widetilde{\mathrm{lk}(a,a)}} \sum_{\mathfrak{s}\in \mathrm{Spin}^c(M)}
    e^{-2\pi i \mathrm{lk}(a,\det(\mathfrak{s}))} \hat Z_\mathfrak{s}(q).
\end{align}
where $\widetilde{\mathrm{lk}(a,a)}=\CS[a]$ is a lift of the value of the linking pairing to $\Q$. For $s=1$ this relation reduces to the one conjectured in \cite{Gukov:2017kmk} and later proved in \cite{Murakami:2023oam}.

\bibliographystyle{JHEP}
\bibliography{main}

\providecommand{\href}[2]{#2}\begingroup\raggedright\begin{thebibliography}{10}

\bibitem{thurston1982three}
W.P.~Thurston, \emph{Three dimensional manifolds, {K}leinian groups and
  hyperbolic geometry}, {\emph{Bulletin of the American Mathematical Society}
  {\bfseries 6} (1982) 357}.

\bibitem{perelman2002entropyformularicciflow}
G.~Perelman, \emph{The entropy formula for the {R}icci flow and its geometric
  applications},  2002.

\bibitem{perelman2003ricciflowsurgerythreemanifolds}
G.~Perelman, \emph{{R}icci flow with surgery on three-manifolds},  2003.

\bibitem{perelman2003finiteextinctiontimesolutions}
G.~Perelman, \emph{Finite extinction time for the solutions to the {R}icci flow
  on certain three-manifolds},  2003.

\bibitem{RT91}
N.~Reshetikhin and V.G.~Turaev, \emph{Invariants of 3-manifolds via link
  polynomials and quantum groups}, {\emph{Inventiones Mathematicae} {\bfseries
  103} (1991) 547}.

\bibitem{turaev1992modular}
V.G.~Turaev, \emph{Modular categories and 3-manifold invariants},
  {\emph{International Journal of Modern Physics B} {\bfseries 6} (1992) 1807}.

\bibitem{Witten:1988hf}
E.~Witten, \emph{{Quantum Field Theory and the Jones Polynomial}},
  \href{https://doi.org/10.1007/BF01217730}{\emph{Commun. Math. Phys.}
  {\bfseries 121} (1989) 351}.

\bibitem{turaev1992state}
V.G.~Turaev and O.Y.~Viro, \emph{State sum invariants of 3-manifolds and
  quantum 6j-symbols}, {\emph{Topology} {\bfseries 31} (1992) 865}.

\bibitem{Witten:1988hc}
E.~Witten, \emph{{(2+1)-Dimensional Gravity as an Exactly Soluble System}},
  \href{https://doi.org/10.1016/0550-3213(88)90143-5}{\emph{Nucl. Phys. B}
  {\bfseries 311} (1988) 46}.

\bibitem{Witten:2010cx}
E.~Witten, \emph{{Analytic Continuation Of Chern-Simons Theory}}, {\emph{AMS/IP
  Stud. Adv. Math.} {\bfseries 50} (2011) 347}
  [\href{https://arxiv.org/abs/1001.2933}{{\ttfamily 1001.2933}}].

\bibitem{kashaev1997hyperbolic}
R.M.~Kashaev, \emph{The hyperbolic volume of knots from the quantum
  dilogarithm}, {\emph{Letters in mathematical physics} {\bfseries 39} (1997)
  269}.

\bibitem{murakami2001colored}
H.~Murakami and J.~Murakami, \emph{The colored {J}ones polynomials and the
  simplicial volume of a knot}, {\emph{Acta Mathematica} {\bfseries 186} (2001)
  85}.

\bibitem{murakami2002kashaev}
H.~Murakami, J.~Murakami, M.~Okamoto, T.~Takata and Y.~Yokota, \emph{Kashaev's
  conjecture and the {C}hern-{S}imons invariants of knots and links},
  {\emph{Experimental Mathematics} {\bfseries 11} (2002) 427}.

\bibitem{zagier2010quantum}
D.~Zagier, \emph{Quantum modular forms}, {\emph{Quanta of Maths} {\bfseries 11}
  (2010) 5}.

\bibitem{Garoufalidis:2021lcp}
S.~Garoufalidis and D.~Zagier, \emph{{Knots, Perturbative Series and Quantum
  Modularity}}, \href{https://doi.org/10.3842/SIGMA.2024.055}{\emph{SIGMA}
  {\bfseries 20} (2024) 055}
  [\href{https://arxiv.org/abs/2111.06645}{{\ttfamily 2111.06645}}].

\bibitem{Chen:2015wfa}
Q.~Chen and T.~Yang, \emph{{Volume conjectures for the
  Reshetikhin{\textendash}Turaev and the Turaev{\textendash}Viro invariants}},
  \href{https://doi.org/10.4171/qt/111}{\emph{Quantum Topol.} {\bfseries 9}
  (2018) 419} [\href{https://arxiv.org/abs/1503.02547}{{\ttfamily
  1503.02547}}].

\bibitem{Wheeler:2023cht}
C.~Wheeler, \emph{{Quantum Modularity for a Closed Hyperbolic 3-Manifold}},
  \href{https://doi.org/10.3842/SIGMA.2025.004}{\emph{SIGMA} {\bfseries 21}
  (2025) 004} [\href{https://arxiv.org/abs/2308.03265}{{\ttfamily
  2308.03265}}].

\bibitem{Gukov:2016gkn}
S.~Gukov, P.~Putrov and C.~Vafa, \emph{{Fivebranes and 3-manifold homology}},
  \href{https://doi.org/10.1007/JHEP07(2017)071}{\emph{JHEP} {\bfseries 07}
  (2017) 071} [\href{https://arxiv.org/abs/1602.05302}{{\ttfamily
  1602.05302}}].

\bibitem{Gukov:2017kmk}
S.~Gukov, D.~Pei, P.~Putrov and C.~Vafa, \emph{{BPS spectra and 3-manifold
  invariants}}, \href{https://doi.org/10.1142/S0218216520400039}{\emph{J. Knot
  Theor. Ramifications} {\bfseries 29} (2020) 2040003}
  [\href{https://arxiv.org/abs/1701.06567}{{\ttfamily 1701.06567}}].

\bibitem{LZ99}
D.~Zagier and R.~Lawrence, \emph{Modular forms and quantum invariants of
  3-manifolds}, {\emph{Asian Journal of Mathematics} {\bfseries 3} (1999) 93}.

\bibitem{Hikami05}
K.~Hikami, \emph{Quantum invariant, modular form, and lattice points},
  {\emph{International Mathematics Research Notices} {\bfseries 2005} (2005)
  121} [\href{https://arxiv.org/abs/arXiv:math-ph/0409016}{{\ttfamily
  arXiv:math-ph/0409016}}].

\bibitem{Hikami05Brieskorn}
K.~Hikami, \emph{On the quantum invariant for the {B}rieskorn homology
  spheres}, {\emph{International Journal of Mathematics} {\bfseries 16} (2005)
  661} [\href{https://arxiv.org/abs/arXiv:math-ph/0405028}{{\ttfamily
  arXiv:math-ph/0405028}}].

\bibitem{Hikami06}
K.~Hikami, \emph{Quantum invariants, modular forms, and lattice points {II}},
  {\emph{Journal of Mathematical Physics} {\bfseries 47} (2006) }
  [\href{https://arxiv.org/abs/arXiv:math/0604091}{{\ttfamily
  arXiv:math/0604091}}].

\bibitem{Hikami06Spherical}
K.~Hikami, \emph{On the quantum invariants for the spherical {S}eifert
  manifolds}, {\emph{Communications in Mathematical Physics} {\bfseries 268}
  (2006) 285} [\href{https://arxiv.org/abs/arXiv:math-ph/0504082}{{\ttfamily
  arXiv:math-ph/0504082}}].

\bibitem{Hikami11}
K.~Hikami, \emph{Decomposition of {W}itten-{R}eshetikhin-{T}uraev invariant:
  Linking pairing and modular forms}, {\emph{AMS/IP Stud. Adv. Math.}
  {\bfseries 50} (2011) 131}.

\bibitem{Gukov:2016njj}
S.~Gukov, M.~Marino and P.~Putrov, \emph{{Resurgence in complex Chern-Simons
  theory}},  \href{https://arxiv.org/abs/1605.07615}{{\ttfamily 1605.07615}}.

\bibitem{Cheng:2018vpl}
M.C.N.~Cheng, S.~Chun, F.~Ferrari, S.~Gukov and S.M.~Harrison, \emph{{3d
  Modularity}}, \href{https://doi.org/10.1007/JHEP10(2019)010}{\emph{JHEP}
  {\bfseries 10} (2019) 010}
  [\href{https://arxiv.org/abs/1809.10148}{{\ttfamily 1809.10148}}].

\bibitem{bringmann2020quantum}
K.~Bringmann, K.~Mahlburg and A.~Milas, \emph{Quantum modular forms and
  plumbing graphs of 3-manifolds}, {\emph{Journal of Combinatorial Theory,
  Series A} {\bfseries 170} (2020) 105145}.

\bibitem{Cheng:2023row}
M.C.N.~Cheng, I.~Coman, D.~Passaro and G.~Sgroi, \emph{{Quantum Modular
  $\widehat Z{}^G$-Invariants}},
  \href{https://doi.org/10.3842/SIGMA.2024.018}{\emph{SIGMA} {\bfseries 20}
  (2024) 018} [\href{https://arxiv.org/abs/2304.03934}{{\ttfamily
  2304.03934}}].

\bibitem{Cheng:2024vou}
M.C.N.~Cheng, I.~Coman, P.~Kucharski, D.~Passaro and G.~Sgroi, \emph{{3d
  Modularity Revisited}},  \href{https://arxiv.org/abs/2403.14920}{{\ttfamily
  2403.14920}}.

\bibitem{murakami1994quantum}
H.~Murakami, \emph{Quantum {SU}(2)-invariants dominate {C}asson's
  {SU}(2)-invariant},  in \emph{Mathematical Proceedings of the Cambridge
  Philosophical Society}, vol.~115, pp.~253--281, Cambridge University Press,
  1994.

\bibitem{lawrence2021witten}
R.~Lawrence, \emph{{W}itten-{R}eshetikhin-{T}uraev invariants of 3-manifolds as
  holomorphic functions},  in \emph{Geometry and Physics}, pp.~363--377, CRC
  Press (2021).

\bibitem{habiro2002quantum}
K.~Habiro, \emph{On the quantum sl2 invariants of knots and integral homology
  spheres}, {\emph{Geom. Topol. Monogr} {\bfseries 4} (2002) 55}.

\bibitem{Freed:1991wd}
D.S.~Freed and R.E.~Gompf, \emph{{Computer calculation of Witten's three
  manifold invariant}}, \href{https://doi.org/10.1007/BF02100006}{\emph{Commun.
  Math. Phys.} {\bfseries 141} (1991) 79}.

\bibitem{andersen2004asymptotic}
J.E.~Andersen, \emph{The asymptotic expansion conjecture},  in \emph{Invariants
  of knots and 3--manifolds (Kyoto 2001): Problems on invariants of knots and
  3-manifolds}, pp.~474--480, Mathematical Sciences Publishers (2004).

\bibitem{andersen2025proofwittensasymptoticexpansion}
J.E.~Andersen, L.~Han, Y.~Li, W.E.~Mistegård, D.~Sauzin and S.~Sun, \emph{{A
  proof of Witten's asymptotic expansion conjecture for WRT invariants of
  Seifert fibered homology spheres}},  2025.

\bibitem{Kontsevich}
M.~Kontsevich, ``{Resurgence from the path integral perspective (Perimeter
  Institute, 2012); Exponential integrals (SCGP and at IHES, 2014 and 2015);
  Resurgence and wall-crossing via complexified path integral (TFC Sendai,
  2016)}.''.

\bibitem{garoufalidis2007chern}
S.~Garoufalidis, \emph{{C}hern-{S}imons theory, analytic continuation and
  arithmetic}, {\emph{arXiv preprint arXiv:0711.1716} (2007) }.

\bibitem{Garoufalidis:2020nut}
S.~Garoufalidis, J.~Gu and M.~Marino, \emph{{The Resurgent Structure of Quantum
  Knot Invariants}},
  \href{https://doi.org/10.1007/s00220-021-04076-0}{\emph{Commun. Math. Phys.}
  {\bfseries 386} (2021) 469}
  [\href{https://arxiv.org/abs/2007.10190}{{\ttfamily 2007.10190}}].

\bibitem{andersen2022resurgence}
J.E.~Andersen and W.E.~Misteg{\aa}rd, \emph{Resurgence analysis of quantum
  invariants of {S}eifert fibered homology spheres}, {\emph{Journal of the
  London Mathematical Society} {\bfseries 105} (2022) 709}.

\bibitem{Gukov:2024vbr}
S.~Gukov and P.~Putrov, \emph{{On categorification of Stokes coefficients in
  Chern-Simons theory}},  \href{https://arxiv.org/abs/2403.12128}{{\ttfamily
  2403.12128}}.

\bibitem{scott1983geometries}
P.~Scott, \emph{{The Geometries of 3-Manifolds}}, {\emph{Bulletin of the London
  Mathematical Society} {\bfseries 15} (1983) 401}.

\bibitem{saveliev2002invariants}
N.~Saveliev, \emph{Invariants for homology 3-spheres}, vol.~140, Springer
  (2002).

\bibitem{brooks1984godbillon}
R.~Brooks and W.~Goldman, \emph{The {G}odbillon-{V}ey invariant of a
  transversely homogeneous foliation}, {\emph{Transactions of the American
  Mathematical Society} {\bfseries 286} (1984) 651}.

\bibitem{brooks1984volumes}
R.~Brooks and W.~Goldman, \emph{Volumes in {S}eifert space}, {\emph{Duke Math.
  J.} {\bfseries 51} (1984) 529}.

\bibitem{khoi2003cut}
V.T.~Khoi, \emph{A cut-and-paste method for computing the {S}eifert volumes},
  {\emph{Mathematische Annalen} {\bfseries 326} (2003) 759}.

\bibitem{KM91}
R.~Kirby and P.~Melvin, \emph{The 3-manifold invariants of {W}itten and
  {R}eshetikhin-{T}uraev for $sl(2, \mathbb{C})$}, {\emph{Inventiones
  Mathematicae} {\bfseries 105} (1991) 473}.

\bibitem{Murakami:2023oam}
Y.~Murakami, \emph{{A Proof of a Conjecture of
  Gukov{\textendash}Pei{\textendash}Putrov{\textendash}Vafa}},
  \href{https://doi.org/10.1007/s00220-024-05136-x}{\emph{Commun. Math. Phys.}
  {\bfseries 405} (2024) 274}
  [\href{https://arxiv.org/abs/2302.13526}{{\ttfamily 2302.13526}}].

\bibitem{LR99}
R.~Lawrence and L.~Rozansky, \emph{{W}itten--{R}eshetikhin--{T}uraev invariants
  of {S}eifert manifolds}, {\emph{Communications in Mathematical Physics}
  {\bfseries 205} (1999) 287}.

\bibitem{NPQ2025}
T.~Nicosanti, P.~Putrov and J.~Quenta-Raygada, \emph{{Approximating
  $\mathrm{SU}(2)$ Chern-Simons theory by finite group gauge theories}},
  \href{https://arxiv.org/abs/2512.11704}{{\ttfamily 2512.11704}}.

\bibitem{Murty17}
M.R.~Murty and S.~Pathak, \emph{Evaluation of the quadratic {G}auss sum},
  {\emph{The Mathematics Student} {\bfseries 86} (2017) 139}.

\bibitem{deloup2007reciprocity}
F.~Deloup and V.~Turaev, \emph{On reciprocity}, {\emph{Journal of Pure and
  Applied Algebra} {\bfseries 208} (2007) 153}.

\bibitem{milnor1973symmetric}
J.W.~Milnor, D.~Husemoller et~al., \emph{Symmetric bilinear forms}, vol.~73,
  Springer (1973).

\bibitem{harichurn2025delta}
S.~Harichurn, J.~Svoboda et~al., \emph{Delta invariants of plumbed manifolds},
  {\emph{SIGMA. Symmetry, Integrability and Geometry: Methods and Applications}
  {\bfseries 21} (2025) 091}.

\bibitem{Kucharski:2019fgh}
P.~Kucharski, \emph{{$\hat{Z}$ invariants at rational $\tau$}},
  \href{https://doi.org/10.1007/JHEP09(2019)092}{\emph{JHEP} {\bfseries 09}
  (2019) 092} [\href{https://arxiv.org/abs/1906.09768}{{\ttfamily
  1906.09768}}].

\bibitem{Chung:2019jgw}
H.-J.~Chung, \emph{{BPS Invariants for 3-Manifolds at Rational Level $K$}},
  \href{https://doi.org/10.1007/JHEP02(2021)083}{\emph{JHEP} {\bfseries 02}
  (2021) 083} [\href{https://arxiv.org/abs/1906.12344}{{\ttfamily
  1906.12344}}].

\bibitem{Costantino:2021yfd}
F.~Costantino, S.~Gukov and P.~Putrov, \emph{{Non-Semisimple TQFT's and BPS
  $q$-Series}}, \href{https://doi.org/10.3842/SIGMA.2023.010}{\emph{SIGMA}
  {\bfseries 19} (2023) 010}
  [\href{https://arxiv.org/abs/2107.14238}{{\ttfamily 2107.14238}}].

\end{thebibliography}\endgroup

\end{document}